\documentclass[english]{amsart}

\usepackage{esint}
\usepackage[svgnames]{xcolor} 
\usepackage[colorlinks,citecolor=red,pagebackref,hypertexnames=false,breaklinks]{hyperref}
\usepackage{pgf,tikz}
\usepackage{pdfsync}

\usepackage{dsfont}
\usepackage{url}
\usepackage[utf8]{inputenc}
\usepackage[T1]{fontenc}
\usepackage{lmodern}
\usepackage{babel}
\usepackage{mathtools}  
\usepackage{amssymb}
\usepackage{lipsum}
\usepackage{mathrsfs}
\usepackage{color}
\usepackage{stmaryrd}
\usepackage{ulem}

\newtheorem{theorem}{Theorem}[section]
\newtheorem{proposition}{Proposition}[section]
\newtheorem{lemma}{Lemma}[section]

\newtheorem{corollary}{Corollary}[section]

\numberwithin{equation}{section}

\title[A quantitative  Borg-Levinson theorem]{A quantitative  Borg-Levinson theorem for a large class of unbounded potentials}

\author[Mourad Choulli]{Mourad Choulli}
\address{Universit\'e de Lorraine}
\email{mourad.choulli@univ-lorraine.fr}

\begin{document}

\begin{abstract}
We prove a quantitative Borg-Levinson theorem for a broad class of unbounded potentials when the spatial dimension is five or greater. In the final section, we outline the modifications required to extend the main result to the anisotropic case.
\end{abstract}



\maketitle

\tableofcontents

\section{Introduction}

Let $\Omega$ be a bounded $C^{1,1}$-domain of $\mathbb{R}^n$, $n\ge 5$,  with boundary $\Gamma$.  Let $p=2n/(n+2)$ and $q=2n/(n+4)$. Their respective conjugates are $p'=2n/(n-2)$ and $q'=2n/(n-4)$. That is we have $1/p+1/p'=1$  and $1/q+1/q'=1$.

Define on $H^1(\Omega)\times H^1(\Omega )$ the sesquilinear form $\mathfrak{a}_V$ associated with $V\in L^{n/2}(\Omega,\mathbb{R} )$ by
\[
\mathfrak{a}_V(u,v)=\int_\Omega \left(\nabla u\cdot \nabla\overline{ v}+Vu\overline{v}\right)dx.
\]
We prove (see Appendix \ref{app1}) that $\mathfrak{a}_V$ is bounded and coercive. Therefore, the bounded operator $A_V: H_0^1(\Omega )\rightarrow H^{-1}(\Omega )$ defined by 
\[
\langle A_Vu,v\rangle =\mathfrak{a}_V(u,v),\quad u,v\in H_0^1(\Omega )
\]
is  self-adjoint (with respect to the pivot space $L^2(\Omega)$) and coercive.  Here, coercivity is to be understood in the sense of the definition given in \cite[(2.7)]{WM}. By applying \cite[Theorem 2.37]{WM}, we deduce that the spectrum of $A_V$, denoted by $\sigma(A_V)$, consists of a sequence $(\lambda_V^k)$ satisfying
\[
-\infty< \lambda_V^1\le \lambda_V^2\le \ldots \le \lambda_V^k\le \ldots
\]
and
\[
\lambda_V^k\rightarrow \infty \quad  \mbox{as}\; k\rightarrow \infty .
\]

Moreover, $L^2(\Omega )$ admits an orthonormal basis $(\phi_V^k)$ of eigenfunctions, each $\phi_V^k$ being associated with $\lambda_V^k$. That is, we have $\phi_V^k\in H_0^1(\Omega)$ and, for all $k\ge 1$ and $v\in H_0^1(\Omega)$,
\[
\int_\Omega \left(\nabla \phi_V^k\cdot \nabla\overline{ v}+V\phi_V^k\overline{v}\right)dx=\lambda_V^k\int_\Omega \phi_V^k\overline{v}dx.
\]
In particular, $(-\Delta +V)\phi_V^k=\lambda_V^k\phi_V^k$ in the distributional sense, and since $(\lambda_V^k-V)\phi_V^k\in L^q(\Omega)$, we obtain from the usual  $W^{2,q}$ regularity that $\phi_V^k\in W^{2,q}(\Omega)$.

Let $\gamma_1$ be the bounded operator from $W^{2,q}(\Omega)$ onto $W^{1-1/q,q}(\Gamma)$ given by
\[
\gamma_1w:=\partial_\nu w,\quad w\in C^\infty(\overline{\Omega}),
\]
where $\partial_\nu$ represents the derivative along the unitary exterior normal vector field $\nu$.

For the sake of simplicity, we use the following notation:
\[
\psi_V^k=\gamma_1 \phi_V^k,\quad k\ge 1,\; V\in L^{n/2}(\Omega,\mathbb{R}).
\]

The following Borg-Levinson type theorem was proved in \cite{Po}.

\begin{theorem}\label{mainthm0}
Let $k_0 \ge 1$ be an integer and $V_1,V_2\in L^m(\Omega,\mathbb{R})$, where $m=n/2$ when $n\ge 4$ and $m>n/2$ when $n=3$, satisfy
\[
\lambda_{V_1}^k=\lambda_{V_2}^k,\quad \psi_{V_1}^k= \psi_{V_2}^k,\quad k\ge k_0.
\]
Then $V_1=V_2$.
\end{theorem}

Let $\Omega_0$ be an arbitrary neighborhood of $\Gamma$ in $\overline{\Omega}$ and $\Gamma_0$ a nonempty open subset of $\Gamma$. Using unique continuation from Cauchy data for elliptic equations, we obtain the following consequence of Theorem \ref{mainthm0}.

\begin{corollary}\label{maincor}
Let $k_0 \ge 1$ be an integer and $V_1,V_2\in L^m(\Omega,\mathbb{R})$, where $m=n/2$ when $n\ge 4$ and $m>n/2$ when $n=3$, satisfy $V_1{_{|\Omega_0}}\in L^n(\Omega_0)$, $V_1=V_2$ in $\Omega_0$ and
\[
\lambda_{V_1}^k=\lambda_{V_2}^k,\quad \psi_{V_1}^k{_{|\Gamma_0}}= \psi_{V_2}^k{_{|\Gamma_0}},\quad k\ge k_0.
\]
Then $V_1=V_2$.
\end{corollary}

Our goal is to establish a quantitative version of Theorem \ref{mainthm0} when $n\ge 5$ and $\ell=1$. Before stating  this result precisely, we need to introduce some definitions and notations. Fix $V_0\in  L^{n/2}(\Omega,\mathbb{R})$ nonnegative and not identically  zero. Then let
\[
\mathscr{V}=\left\{V\in  L^{n/2}(\Omega,\mathbb{R});\; |V|\le V_0\right\}.
\]

We fix $0<\varepsilon<\frac{1}{q'}$. Let $t=1/q+1+\varepsilon$ and $2<r \le p'$ be given  such that $\varsigma:=\frac{n(r-2)}{2r}<\frac{1}{q'}-\varepsilon$.
Then  set 
\[
\mathfrak{a}:=\frac{2r}{r-2}\, (\ge n),\quad \upsilon :=1-\frac{t}{2}-\frac{\varsigma}{2}\, (>0),\quad \sigma:=\frac{3}{\upsilon},
\]
and
\[
 \beta:=\frac{1}{(n+2)(\sigma(\ell +1)+n+9/2)}.
\]
Here and throughout the remainder of this text, $\ell$ will denote the smallest integer satisfying $\ell > n/2+1$. Define, for $V_1,V_2\in \mathscr{V}$,

\[
\mathscr{D} (V_1,V_2)=\sum_{k\ge 1}k^{-2\ell/n}\left[ |\lambda_{V_1}^k-\lambda_{V_2}^k|+\|\psi_{V_2}^k-\psi_{V_1}^k\|_{L^q(\Gamma)}\right].
\]
Given Weyl's asymptotic formula \eqref{1.6},  inequality \eqref{1.33}, and the continuity of the trace map $w\in W^{2,t}(\Omega)\mapsto \gamma_1w\in L^q(\Gamma)$, we verify that the series above converges.

Let $W_0\in L^\mathfrak{a}(\Omega,\mathbb{R})$ be nonnegative and not identically zero, and set
\[
\mathscr{W}=\{(V_1,V_2)\in \mathscr{V}\times \mathscr{V};\; |V_1-V_2|\le W_0 \}.
\]

\begin{theorem}\label{mainthm1}
For all $(V_1,V_2)\in \mathscr{W}$, we have
\begin{equation}\label{mt1}
\|V_1-V_2\|_{H^{-1}(\Omega)}\le \mathbf{c}\mathscr{D}(V_1,V_2)^\beta,
\end{equation}
where $\mathbf{c}=\mathbf{c}(n,\Omega,V_0,W_0,\varepsilon, r,\ell)>0$ is a constant.
\end{theorem}

For the sake of clarity, we have deliberately excluded $n\in \{2,3,4\}$ because certain modifications are required in that case. This will be the subject of future work.

In the literature, the sequence $(\lambda_V^k,\psi_V^k)_{k\ge 1}$ (resp. $(\lambda_V^k,\psi_V^k)_{k\ge \ell}$) associated with $V\in L^{n/2}(\Omega,\mathbb{R})$ is referred to as complete boundary spectral data (resp. incomplete boundary spectral data).

The first uniqueness result for potentials in $L^\alpha$ (with $\alpha > n/2$) is due to Päivärinta and Serov \cite{PS} for the case of complete boundary spectral data. This result was generalized by Pohjola \cite{Po}, as stated in Theorem \ref{mainthm0}. The case of the Robin boundary condition was recently studied by the author, Metidji, and Soccorsi \cite{CMS} for potentials belonging to the same class as in Theorem \ref{mainthm0}, but with incomplete boundary spectral data--specifically, data consisting of $((\lambda_V^k)_{k\ge \ell},(\psi_V^k)_{k\ge 1})$. Bellassoued, Kian, Mannoubi, and Soccorsi \cite{BKMS} established a result for potentials in $L^{\max(2,3n/5)}$ based on knowledge of asymptotic spectral data. For a precise definition of asymptotic spectral data, we refer to \cite{KKS}, where this concept was first introduced. A quantitative version of the uniqueness result from \cite{BKMS} appears in \cite{KS}. A stability inequality based on asymptotic spectral data for the Robin boundary condition was obtained by the author, Metidji, and Soccorsi in \cite{CMS2}.

The first multidimensional Borg-Levinson-type theorem was proved by Nachman, Sylvester, and Uhlmann \cite{NSU} by establishing the relationship between complete boundary spectral data and the Dirichlet-to-Neumann map. The quantitative version of this result was proved by Alessandrini and Sylvester \cite{AS} by linking the original problem to a hyperbolic Dirichlet-to-Neumann map (see also \cite{Ch}, where the stability inequality is reformulated in an appropriate topology). The generalization of the result in \cite{NSU} to the case of incomplete boundary spectral data is due to Isozaki \cite{Isozaki89, I}; he employed ideas drawn from the Born approximation.

The magnetic case was first studied by Kian in \cite{Ki} for the isotropic setting and subsequently generalized to the anisotropic case in \cite{BCDKS}. For further results, we refer--without claiming to be exhaustive--to \cite{BCY,BCY2,BD,Ch2,Ser,So}.

We recall that the  quantitative version of the multidimensional Borg-Levinson theorem established by Alessandrini and Sylvester \cite{AS} uses a similar result stated in Corollary \ref{corDN1}, combined with a quantitative version of the determination of bounded potential from the corresponding hyperbolic Dirichlet-to-Neumann. In \cite{CS}, we introduced a new method to generalize the result of \cite{AS} to the same problem with more general spectral data. This method does not relies on hyperbolic Dirichlet-to-Neumann map. Another method has been proposed by Kavian, Kian and Soccorsi, but this later assume that $\sum_{k\ge 1}\|\psi_{V_2}^k-\psi_{V_1}^k\|_{L^2(\Gamma)}^2<\infty$.
Our proof of Theorem \ref{mainthm1} is inspired by the work of the author and Stefanov \cite{CS} for the case of bounded potentials. Specifically, we use the non-Hilbertian part of \cite{CS}, which relies essentially on the use of Taylor's formula at a sufficiently high order.

\section{Preliminaries}

This section presents the preliminary results that will be used in the following section to prove Theorem \ref{mainthm1}. First, since $H^1(\Omega)$ is continuously embedded in $L^{p'}(\Omega)$, we can define
 \[
 \mathbf{e}=\sup\left\{\|w\|_{L^{p'}(\Omega)};\; u\in H^1(\Omega), \; \|u\|_{H^1(\Omega)}=1\right\}.
 \]
The space $H_0^1(\Omega)$ will be endowed with the norm $\|\nabla \cdot\|_{L^2(\Omega)}$, and we set 
\[
\mathbf{p}=\sup \left\{ \|u\|_{H^1(\Omega )};\; u\in H_0^1(\Omega),\; \|\nabla u\|_{L^2(\Omega )}=1\right\}.
\]

\subsection{Weyl's asymptotic  formula}

We  have the following improvement of \cite[Lemma 5.2]{Po}.

\begin{lemma}\label{lem3}
Let $V\in \mathscr{V}$ and $u\in H_0^1(\Omega)$. Then, for all $\epsilon>0$, there exists a constant $\mathbf{c}_\epsilon=\mathbf{c}_\epsilon (n,\Omega,V_0,\epsilon)>0$ such that 
\begin{equation}\label{1.12.0}
\|Vu^2\|_{L^1(\Omega)}\le \epsilon \|\nabla u\|_{L^2(\Omega)}^2+\mathbf{c}_\epsilon\|u\|_{L^2(\Omega)}^2.
\end{equation}
\end{lemma}

\begin{proof}
By applying H\"older's inequality, we obtain
\begin{align*}
\|Vu^2\|_{L^1(\Omega)}\le \|V_0u^2\|_{L^1(\Omega)} &\le \|V_0u\|_{L^p(\Omega)}\|u\|_{L^{p'}(\Omega)}
\\
&\le \mathbf{p}\mathbf{e}\|V_0u\|_{L^p(\Omega)}\|\nabla u\|_{L^2(\Omega)}.
\end{align*}
Consequently, for all $\epsilon >0$, we have 
\begin{equation}\label{1.12.1}
\|Vu^2\|_{L^1(\Omega)}\le 2\mathbf{p}\mathbf{e}\epsilon\|\nabla u\|_{L^2(\Omega)}^2+2\mathbf{p}\mathbf{e}\epsilon^{-1}\|V_0u\|_{L^p(\Omega)}^2.
\end{equation}
On the other hand, there exists a constant $\mathbf{c_0}=\mathbf{c}_0(n,\Omega)>0$ such that, for all $U_0\in C_0^\infty (\Omega)$, we have
\begin{align*}
\|V_0u\|_{L^p(\Omega)}^2&\le 2\|U_0u\|_{L^p(\Omega)}^2+2\|(V_0-U_0)u\|_{L^p(\Omega)}^2
\\
&\le 2\|U_0\|_{L^\infty (\Omega)}^2\|u\|_{L^p(\Omega)}^2+2\|(V_0-U_0)\|_{L^{n/2}(\Omega)}^2\|u\|_{L^{p'}(\Omega)}^2
\\
&\le \mathbf{c_0}\|U_0\|_{L^\infty (\Omega)}^2\|u\|_{L^2(\Omega)}^2+\mathbf{c_0}\|(V_0-U_0)\|_{L^{n/2}(\Omega)}^2\|u\|_{L^{p'}(\Omega)}^2
\\
&\le \mathbf{c_0}\|U_0\|_{L^\infty (\Omega)}^2\|u\|_{L^2(\Omega)}^2+\mathbf{c_0}\|(V_0-U_0)\|_{L^{n/2}(\Omega)}^2\|\nabla u\|_{L^2(\Omega)}^2.
\end{align*}
Since $C_0^\infty (\Omega)$ is dense in $L^{n/2}(\Omega)$, we choose $U_0$ so as to obtain
\begin{equation}\label{1.12.2}
\|V_0u\|_{L^p(\Omega)}^2\le \mathbf{c_0}\|W_0\|_{L^\infty (\Omega)}^2\|u\|_{L^2(\Omega)}^2+\mathbf{c_0}\epsilon^2\|\nabla u\|_{L^2(\Omega)}^2.
\end{equation}
This and \eqref{1.12.1} imply
\begin{equation}\label{1.12.3}
\|Vu^2\|_{L^1(\Omega)}\le 2(\mathbf{c}_0+1)\mathbf{p}\mathbf{e}\epsilon \|\nabla u\|_{L^2(\Omega)}^2+2\mathbf{c}_0\mathbf{p}\mathbf{e}\epsilon^{-1}\|W_0\|_{L^\infty (\Omega)}^2\|u\|_{L^2(\Omega)}^2,
\end{equation}
from which we deduce the expected inequality.
\end{proof}

\begin{proposition}\label{pro1}
There exist two constants $\mathfrak{c}_0=\mathfrak{c}_0(n,\Omega)\ge 1$ and $\mathfrak{c}_1=\mathfrak{c}_1(n,\Omega,V_0)>0$ such that, for all $V\in \mathscr{V}$, we have
\begin{equation}\label{1.6}
\mathfrak{c}_0^{-1}k^{2/n}-\mathfrak{c}_1\le  \lambda_V^k \le \mathfrak{c}_0k^{2/n}+\mathfrak{c}_1,\quad k\ge 1.
\end{equation}
\end{proposition}

\begin{proof}
In this proof, $\mathfrak{c}_1=\mathfrak{c}_1(n,\Omega,V_0)>0$ will denote a generic constant. Let $V\in \mathscr{V}$. By applying \eqref{1.12.0} with $\epsilon=1$, we find
\begin{equation}\label{1.12.4}
\mathfrak{a}_V(u,u)\le 2\|\nabla u\|_{L^2(\Omega)}+\mathfrak{c}_1\|u\|_{L^2(\Omega)},\quad u\in H_0^1(\Omega).
\end{equation}
On the other hand, we have
\[
\mathfrak{a}_V(u,u)\ge \mathfrak{a}_{-V_0}(u,u)= \|\nabla u\|_{L^2(\Omega)}^2-\|V_0u^2\|_{L^1(\Omega)},\quad u\in H_0^1(\Omega).
\]
By applying again \eqref{1.12.0} with $\epsilon=1/2$, we get
\[
\mathfrak{a}_V(u,u)\ge \frac{1}{2}\|\nabla u\|_{L^2(\Omega)}^2-\mathfrak{c}_1\|u\|_{L^2(\Omega)}^2,\quad u\in H_0^1(\Omega).
\]
Let $( \mu_k)$ be the nondecreasing sequence of eigenvalues of $\mathfrak{a}_0$ (that is, $\mathfrak{a}_V$ with $V=0$). Then the min-max principle yields
\[
\frac{\mu_k}{2}-\mathfrak{c}_1\le \lambda_V^k \le 2\mu_k +\mathfrak{c}_1,\quad k\ge 1.
\]
We complete the proof by using Weyl's classical asymptotic formula, according to which $\mu_k \underset{k\rightarrow \infty}{\sim} k^{2/n}$.
\end{proof}

\subsection{Dirichlet-to-Neumann map}

Let $\rho (A_V)=\mathbb{C}\setminus\sigma(A_V)$, usually called the resolvent set of  $A_V$.  We consider the  BVP
\begin{equation}\label{1.7.e}
\left\{
\begin{array}{ll}
(-\Delta +V- \lambda)u=0\quad &\mbox{in}\; \Omega,
\\
u=\varphi &\mbox{on}\; \Gamma .
\end{array}
\right.
\end{equation}

Let $V\in L^{n/2}(\Omega,\mathbb{R} )$ and $\lambda \in \rho (A_V)$. It follows from Theorem \ref{thm2} in Appendix \ref{app2} that, for all $\varphi\in W^{2-1/p,p}(\Gamma )$, the BVP \eqref{1.7.e} admits a unique solution $u_V(\lambda)(\varphi )\in W^{2,p}(\Omega )$. Furthermore, we have
\begin{equation}\label{1.8.e}
\|u_V(\lambda)(\varphi)\|_{W^{2,p}(\Omega )}\le \varkappa\|\varphi\|_{W^{2-1/p,p}(\Gamma )},
\end{equation}
where $\varkappa=\varkappa(n,\Omega,V,\lambda)>0$ is a constant.

Define the family of Dirichlet-to-Neumann maps $(\Lambda_V(\lambda ))_{\lambda \in \rho (A_V)}$ associated with $V\in L^{n/2}(\Omega,\mathbb{R} )$ as follows:
\[
\Lambda_V(\lambda ): \varphi\in W^{2-1/p,p}(\Gamma ) \mapsto \gamma_1 u_V(\lambda)(\varphi) \in W^{1-1/p,p}(\Gamma ).
\]
By the continuity of the trace map $w\in W^{2,p}(\Omega)\mapsto \gamma_1w\in W^{1-1/p,p}(\Gamma)$, it follows from \eqref{1.8.e} that  the mapping $\Lambda_V(\lambda )$ defines a bounded operator.

\subsection{Analyticity of solutions}

In this subsection, $\mathbf{c}=\mathbf{c}(n,\Omega,V_0)>0$ will denote a generic constant. Let $V\in \mathscr{V}$ and $k\ge 1$. Since 
\[
\|\nabla \phi_V^k\|_{L^2(\Omega)}^2\le |\lambda_V^k|+\|V|\phi_V^k|^2\|_{L^1(\Omega)},
\]
\eqref{1.12.0} with $\epsilon=1/2$ yields
\[
\|\nabla \phi_V^k\|_{L^2(\Omega)}\le \mathbf{c}(1+|\lambda_V^k|)^{1/2}
\]
and then
\begin{equation}\label{1.33.0}
\|\phi_V^k\|_{L^{p'}(\Omega)}\le \mathbf{c}(1+|\lambda_V^k|)^{1/2}.
\end{equation}
Consequently, we have
\begin{equation}\label{1.33.1}
\|V\phi_V^k\|_{L^p(\Omega)}\le \mathbf{c}(1+|\lambda_V^k|)^{1/2}.
\end{equation}

By using that $-\Delta \phi_V^k=(\lambda_V^k-V)\phi_V^k\in L^p(\Omega)\subset L^q(\Omega)$ (with continuous embedding), we get $\phi_V^k\in W^{2,q}(\Omega)$. In view of the $W^{2,q}$ a priori estimate (e.g. \cite[Theorem 2.3.3.6]{Gr}), \eqref{1.33.1} yields
\begin{equation}\label{1.33}
\|\phi_V^k\|_{W^{2,q}(\Omega)}\le \mathbf{c}(1+|\lambda_V^k|).
\end{equation}

Let $\lambda \in \rho(A_V)$, $\varphi\in W^{2-1/p,p}(\Gamma)$ and $u:=u_V(\lambda)(\varphi)$. Then
\[
0=\int_\Omega \left(\nabla u\cdot \nabla \overline{\phi_V^k} + (V-\lambda) u\overline{\phi_V^k}\right)dx,
\]
where we used that $\nabla u\in L^2(\Omega,\mathbb{C}^n)$ and $u\in L^{p'}(\Omega)$.

Let $(\cdot, \cdot)$ denote the inner product in $L^2(\Omega)$. By performing integration by parts, we deduce from the preceding identity
\begin{equation}\label{1.34}
(u,\phi_V^k)=-\frac{\langle \varphi, \psi_V^k\rangle}{\lambda_V^k-\lambda}.
\end{equation}
Here and henceforth,
\[
\langle \varphi, \psi_V^k\rangle :=\int_\Gamma \varphi\overline{\psi_V^k}d\sigma.
\]
Inequality \eqref{1.34} then yields
\[
u_V(\lambda)(\varphi)=-\sum_{k\ge 1}\frac{\langle \varphi, \psi_V^k\rangle}{\lambda_V^k-\lambda}\phi_V^k.
\]
In particular, we have
\begin{equation}\label{1.35}
\sum_{k\ge 1}\left|\frac{\langle \varphi, \psi_V^k\rangle}{\lambda_V^k-\lambda}\right|^2=\|u_V(\lambda)(g)\|_{L^2(\Omega)}^2.
\end{equation}

Let $\lambda_0\in \rho(A_V)$ and $d=\mathrm{dist}(\lambda_0,\sigma(A_V))$. Since 
\[
|\lambda_k-\lambda|\ge |\lambda_k-\lambda_0|-d/2,\quad \lambda\in B(\lambda_0,d/2),
\]
and $|\lambda_k-\lambda_0|-d/2\underset{k \to \infty}{\sim} |\lambda_k-\lambda_0|$, \eqref{1.35} shows that the series $\sum_{k\ge 1}\frac{\langle \varphi, \psi_V^k\rangle}{\lambda_V^k-\lambda}\phi_V^k$ converges in $L^2(\Omega)$, uniformly on $B(\lambda_0,d/2)$, and thus on any compact subset of $\rho(A_V)$. Weierstrass's theorem then allows us to conclude that the mapping
\[
\lambda \in \rho(A_V)\mapsto u_V(\lambda)(\varphi)\in L^2(\Omega)
\]
is holomorphic and 
\begin{equation}\label{1.36}
 u_V^{(j)}(\lambda)(\varphi)= -j!\sum_{k\ge 1}\frac{\langle \varphi, \psi_V^k\rangle}{(\lambda_V^k-\lambda)^{j+1}}\phi_V^k,\quad j\ge 0,
\end{equation}
where we have used the notation
\[
w^{(j)}:=\frac{d^j}{d\lambda^j}w,
\]
which we will retain in the sequel.

In fact, we have $\lambda \in \rho(A_V)\mapsto u_V(\lambda)(\varphi)\in W^{2,p}(\Omega)$ is holomorphic as shown by the following proposition.

\begin{proposition}\label{prohol}
For all $V\in \mathscr{V}$ and  $\varphi\in W^{2-1/p,p}(\Gamma)$, \[\lambda \in \rho(A_V)\mapsto u_V(\lambda)(\varphi)\in W^{2,p}(\Omega)\] is holomorphic and $u^{(1)}(\lambda):=u_V^{(1)}(\lambda)(\varphi)$ is the solution of the BVP
\begin{equation}\label{holo1}
(-\Delta +V-\lambda )u^{(1)}(\lambda)=u_V(\lambda)(\varphi) \; \mathrm{in}\;  \Omega,\quad u=0\; \mathrm{on}\; \Gamma.
\end{equation}
\end{proposition}

\begin{proof}
Let $V\in \mathscr{V}$, $\lambda \in \rho(A_V)$ and $\mu\in B(0,\mathrm{dist}(\lambda,\sigma(A_V)/2))$. From Theorem \ref{thm2} and its proof, the constant $\mathbf{c}_{\lambda+\mu}$ in \eqref{1.8} can be chosen such that 
\[
\mathbf{c}_\lambda:=\sup_{\mu\in B(0,\mathrm{dist}(\lambda,\sigma(A_V)/2)}\mathbf{c}_{\lambda+\mu}<\infty .
\]
Since $v=u(\lambda+\mu)(\varphi)-u(\lambda)(\varphi)$ is the solution of the BVP
\[
(-\Delta +V-(\lambda +\mu))v=\mu u(\lambda)(\varphi) \; \mathrm{in}\;  \Omega,\quad v=0\; \mathrm{on}\; \Gamma,
\]
we obtain from \eqref{1.8}
\begin{equation}\label{holo2}
\|u(\lambda+\mu)(\varphi)-u(\lambda)(\varphi)\|_{W^{2,p}(\Omega)}\le \mathbf{c}_\lambda |\mu|\|u(\lambda)(\varphi)\|_{L^p(\Omega)}.
\end{equation}
Hence, $z \in \rho(A_V)\mapsto u_V(z)(\varphi)\in W^{2,p}(\Omega)$ is continuous.

Next, let $u^{(1)}(\lambda)$ be the solution of \eqref{holo1} and set
\[
w:=u(\lambda+\mu)(\varphi)-u(\lambda)(\varphi)-\mu u^{(1)}(\lambda).
\]
We verify that $w$ is the solution of the BVP
\[
(-\Delta +V-\lambda )w=\mu [u(\lambda+\mu)(\varphi)-u(\lambda)(\varphi)] \; \mathrm{in}\;  \Omega,\quad w=0\; \mathrm{on}\; \Gamma.
\]
By applying \eqref{1.8} and then \eqref{holo2}, we obtain
\[
\|u(\lambda+\mu)(\varphi)-u(\lambda)(\varphi)-\mu u^{(1)}(\lambda)\|_{W^{2,p}(\Omega)}\le \mathbf{c}_\lambda^2 |\mu|^2\|u(\lambda)(\varphi)\|_{L^p(\Omega)}.
\]
This inequality gives to the expect result.
\end{proof}

Since $W^{2,p}(\Omega)$ embeds continuously into $W^{2,q}(\Omega)$, the mapping $\lambda \in \rho(A_V) \mapsto u_V(\lambda)(\varphi) \in W^{2,q}(\Omega)$ is also holomorphic for all $\varphi \in W^{2-1/p,p}(\Gamma)$. This observation will be used later.
 
Let $V\in \mathscr{V}$,  $\varphi\in W^{2-1/p,p}(\Gamma)$ and $\lambda \in \rho(A_V)$. By induction in $j\ge 1$, we derive from \eqref{holo1}
\begin{equation}\label{d}
(-\Delta +V-\lambda)u_V^{(j)}(\lambda)(\varphi) =ju_V^{(j-1)}(\lambda)(\varphi)\; \mathrm{in}\;\Omega,\quad u^{(j)}(\lambda)=0\; \mathrm{on}\; \Gamma.
\end{equation}

\subsection{Resolvent estimates}

Recall that the resolvent associated with the operator $A_V$ is defined as follows:
\[
R_V(\lambda )=(A_V-\lambda )^{-1}: H^{-1}(\Omega) \rightarrow H_0^1(\Omega), \quad \lambda \in \rho (A_V).
\]

For $z \in \mathbb{C}$, $\Re z$ and $\Im z$ will denote the real and imaginary parts of $z$, respectively. Let 
\[
\Sigma_0=\{\lambda \in \mathbb{C};\; |\Re \lambda|\ge 1 ,\; |\Im \lambda| \ge 1,\; |\Re \lambda| |\Im \lambda|^{-1}\ge 1/2\}. 
\]

In all of this subsection, $\mathbf{c}=\mathbf{c}(n,\Omega,V_0)>0$ will denote a generic constant.

\begin{lemma}\label{lem4}
Let $V\in  \mathscr{V}$, $\lambda \in \Sigma_0$ and $f\in L^2(\Omega)$. Then
\begin{align}
&\|R_V(\lambda)f\|_{L^2(\Omega)}\le|\Im \lambda|^{-1}\|f\|_{L^2(\Omega)}, \label{1.13}
\\
&\|R_V(\lambda)f\|_{H_0^1(\Omega)}\le \mathbf{c} |\Re \lambda|^{1/2}|\Im \lambda|^{-1}\|f\|_{L^2(\Omega)},\label{1.14.0}
\\
&\|R_V(\lambda)f\|_{H_0^1(\Omega)}\le \mathbf{c}|\Re \lambda||\Im \lambda|^{-1}\|f\|_{L^p(\Omega)}.\label{1.14}
\end{align}
\end{lemma}

\begin{proof}
Let $u:=R_V(\lambda)f$. Green's formula then yields
\[
\int_\Omega|\nabla u|^2dx -\int_\Omega (\lambda -V)|u|^2dx=\int_\Omega f\overline{u}dx.
\]
Hence,
\begin{align}
&\int_\Omega|\nabla u|^2dx -\int_\Omega(\Re\lambda -V)|u|^2dx=\Re \int_\Omega f\overline{u}dx, \label{1.15}
\\
& -\Im \lambda \int_\Omega |u|^2dx=\Im \int_\Omega f\overline{u}dx.\label{1.16}
\end{align}
Clearly, \eqref{1.13} follows from \eqref{1.16}.

Next, by using \eqref{1.15} and \eqref{1.12.0} with $\epsilon=1/2$, we  get
\[
\| \nabla u\|_{L^2(\Omega)}^2\le \mathbf{c}|\Re \lambda|\|u\|_{L^2(\Omega)}^2+\|f\overline{u}\|_{L^1(\Omega)},
\]
which, combined with \eqref{1.16}, yields
\begin{align}
\| \nabla u\|_{L^2(\Omega)}^2&\le \mathbf{c}|\Re \lambda||\Im \lambda|^{-1}\|f\overline{u}\|_{L^1(\Omega)} \label{1.17}
\\
&\le \mathbf{c}|\Re \lambda| |\Im \lambda|^{-1}\|u\|_{L^2(\Omega)}\|f\|_{L^2(\Omega)}.\nonumber
\end{align}
Putting together \eqref{1.13} and \eqref{1.17}, we obtain
\[
\| \nabla u\|_{L^2(\Omega)}^2\le \mathbf{c} |\Re \lambda||\Im \lambda|^{-2}\|f\|_{L^2(\Omega)}^2.
\]
That is we proved \eqref{1.14.0}.

On the other hand, applying H\"older's inequality, we find
\[
\|f\overline{u}\|_{L^1(\Omega)}\le \|u\|_{L^{p'}(\Omega)}\|f\|_{L^p(\Omega)},
\]
and hence, for all $\epsilon >0$, 
\[
\|f\overline{u}\|_{L^1(\Omega)}\le\frac{\epsilon}{2}\|u\|_{L^{p'}(\Omega)}^2+\frac{1}{2\epsilon}\|f\|_{L^p(\Omega)}^2.
\]
Therefore,
\begin{equation}\label{1.18}
\|f\overline{u}\|_{L^1(\Omega)}\le\frac{\epsilon \mathbf{e}\mathbf{p}}{2}\| \nabla u\|_{L^2(\Omega)}^2+\frac{1}{2\epsilon}\|f\|_{L^p(\Omega)}^2.
\end{equation}
By using \eqref{1.18} in the first inequality of \eqref{1.17}, we obtain
\[
\| \nabla u\|_{L^2(\Omega)}^2
\le c_0|\Re \lambda||\Im \lambda|^{-1}\left[\frac{\epsilon \mathbf{e}\mathbf{p}}{2}\| \nabla u\|_{L^2(\Omega)}^2+\frac{1}{2\epsilon}\|f\|_{L^p(\Omega)}^2\right],
\]
where $c_0>0$ is a universal constant. Choosing $\epsilon$ in this inequality such that $c_0|\Re \lambda||\Im \lambda|^{-1}\epsilon \mathbf{e}\mathbf{p}=1$, we obtain \eqref{1.14}.
\end{proof}

Define
\[
\Sigma=\{\lambda =(\tau+i)^2;\; \tau \ge 2\}\; (\subset \Sigma_0).
\]

\begin{corollary}\label{cor1}
Let $V\in  \mathscr{V}$, $\lambda=(\tau+i)^2\in \Sigma$ and $f\in L^2(\Omega)$. Then
\begin{align}
&2\|R_V(\lambda)f\|_{L^2(\Omega)}\le \tau^{-1}\|f\|_{L^2(\Omega)}, \label{1.19}
\\
&\|R_V(\lambda)f\|_{L^{p'}(\Omega)}\le \mathbf{c}\tau \|f\|_{L^p(\Omega)} .\label{1.20}
\end{align}
\end{corollary}

Define
\begin{equation}\label{theta}
p_\theta=2n/(n+2\theta),\quad p_\theta '=2n/(n-2\theta),\quad \theta\in [0,1],
\end{equation}
and note that  $p_\theta '$ is the conjugate of $p_\theta$. By applying the Riesz-Thorin theorem, we obtain the following result from \eqref{1.19} and \eqref{1.20}.

\begin{proposition}\label{prort}
Let $V\in  \mathscr{V}$. Then, for all $\lambda =(\tau+i)^2\in \Sigma$, $R_V(\lambda)$ maps continuously $L^{p_\theta}(\Omega)$ into $L^{p_\theta '}(\Omega)$ and the following estimate holds:
\begin{equation}\label{1.21}
\|R_V(\lambda)f\|_{L^{p_\theta '}(\Omega)}\le \mathbf{c}\tau^{-1+2\theta}\|f\|_{L^{p_\theta }(\Omega)},\quad f\in L^{p_\theta}(\Omega).
\end{equation}
\end{proposition}

\begin{lemma}\label{lem4.1}
Let $V\in  \mathscr{V}$. There exists $\lambda_\ast=\lambda_\ast(n,\Omega,V_0)\ge 2$ such that for any $\lambda \in \mathbb{C}\setminus \mathbb{R}$ satisfying $-\Re \lambda \ge \lambda_\ast$ and $f\in L^2(\Omega)$ we have
\begin{align}
&\|R_V(\lambda)f\|_{L^2(\Omega)}\le 2|\Re \lambda|^{-1}\|f\|_{L^2(\Omega)}, \label{re1}
\\
&\|R_V(\lambda)f\|_{H_0^1(\Omega)}\le 2|\Re \lambda|^{-1/2}\|f\|_{L^2(\Omega)}.\label{re2}
\end{align}
\end{lemma}

\begin{proof}
Let $\lambda \in \mathbb{C}\setminus \mathbb{R}$ with $\Re \lambda \le 0$, $f\in L^2(\Omega)$  and $u:=R_V(\lambda)f$. From \eqref{1.15}, we obtain
\begin{equation}\label{r1}
\int_\Omega|\nabla u|^2dx +|\Re \lambda |\int_\Omega |u|^2dx=-\int_\Omega V|u|^2dx+\Re \int_\Omega f\overline{u}dx,
\end{equation}
and from \eqref{1.12.0}, we get
\begin{equation}\label{r2}
\|Vu^2\|_{L^1(\Omega)}\le \frac{1}{2} \|\nabla u\|_{L^2(\Omega)}^2+\mathbf{c}\|u\|_{L^2(\Omega)}^2.
\end{equation}
On the other hand, we have
\begin{equation}\label{r3}
\left|\Re \int_\Omega f\overline{u}dx\right|\le \|u\|_{L^2(\Omega)}\|f\|_{L^2(\Omega)}.
\end{equation}
Combining \eqref{r1}, \eqref{r2} and \eqref{r3}, we obtain
 \begin{equation}\label{r4}
\int_\Omega|\nabla u|^2dx +2(|\Re \lambda| -\mathbf{c}) \int_\Omega |u|^2dx\le 2\|u\|_{L^2(\Omega)}\|f\|_{L^2(\Omega)}.
\end{equation}
Let $\lambda_\ast=2\mathbf{c}$. If $|\Re \lambda|\ge \lambda_\ast$, then \eqref{r4} yields
\begin{align*}
&\|u\|_{L^2(\Omega)}\le 2|\Re \lambda|^{-1}\|f\|_{L^2(\Omega)},
\\
&\|\nabla u\|_{L^2(\Omega)}\le 2|\Re \lambda|^{-1/2}\|f\|_{L^2(\Omega)}.
\end{align*}
These are the expected inequalities.
\end{proof}

From \cite[Theorem 9.14]{GT}, there exists $c_0=c_0(n,\Omega)>0$ and $\mu_0=\mu_0(n,\Omega)>0$ such that for any $\mu \ge \mu_0$ and $w\in W^{2,q}(\Omega)\cap W_0^{1,q}(\Omega)$ we have
\begin{equation}\label{ae1}
\|w\|_{W^{2,q}(\Omega)}\le c_0 \|(\Delta -\mu)w \|_{L^q(\Omega)}.
\end{equation}

Define
\[
\Pi_1=\{\lambda \in \mathbb{C}\setminus \mathbb{R};\; -\Re \lambda \ge \max(\mu_0,\lambda_\ast),\; |\Im \lambda||\Re \lambda|^{-1}\le 1 \},
\]
where $\mu_0$ is as above and $\lambda^\ast$ is as in Lemma \ref{lem4.1}.

\begin{lemma}\label{lemres}
Let $V\in \mathscr{V}$, $f\in L^2(\Omega)$ and $\lambda \in \Pi_1$. Then
\begin{equation}\label{3.1}
\|R_V(\lambda)(f)\|_{W^{2,q}(\Omega)}\le \mathbf{c}\|f\|_{L^2(\Omega)}.
\end{equation}
\end{lemma}

\begin{proof}
Let $u:=R_V(\lambda)f\in W^{2,q}(\Omega) \cap W_0^{1,q}(\Omega)$, $v:=\Re u$ and $w:=\Im u$.  As $-\Delta u=(\lambda -V)u+f$, we obtain 
 \[
 (\Delta +\Re \lambda) v=  \Im \lambda w+Vv-\Re f,\quad (\Delta+\Re \lambda) w= -\Im \lambda v+Vw-\Im f.
 \]
Then \eqref{ae1} yields
\begin{align*}
\|v\|_{W^{2,q}(\Omega)}&\le c_0 \|(\Delta +\Re \lambda) v\|_{L^q(\Omega)}
\\
&\le c_0\|\Im \lambda w+Vv-\Re f\|_{L^q(\Omega)}.
\end{align*}
Since
\[
\|Vv\|_{L^q(\Omega)}\le \|V_0\|_{L^{n/2}(\Omega)}\|v\|_{L^2(\Omega)},
\]
we obtain 
\[
\|v\|_{W^{2,q}(\Omega)}\le \mathbf{c}\left(|\Im \lambda|\|w\|_{L^2(\Omega)}+\|v\|_{L^2(\Omega)}+\|f\|_{L^2(\Omega)} \right).
\]
Similarly, we have
\[
\|w\|_{W^{2,q}(\Omega)}\le \mathbf{c}\left(|\Im \lambda|\|v\|_{L^2(\Omega)}+\|w\|_{L^2(\Omega)}+\|f\|_{L^2(\Omega)} \right).
\]
By combining the last two inequalities, we find
\[
\|u\|_{W^{2,q}(\Omega)}\le\mathbf{c}\left(|\Im \lambda|\|u\|_{L^2(\Omega)}+\|u\|_{L^2(\Omega)}+\|f\|_{L^2(\Omega)} \right).
\]
This and \eqref{re1} imply
\[
\|u\|_{W^{2,q}(\Omega)}\le \mathbf{c}\left(|\Im \lambda||\Re \lambda|^{-1}+1\right)\|f\|_{L^2(\Omega)},
\]
from which \eqref{3.1} follows.
\end{proof}

It follows from  \cite[Theorem 2.3.3.6]{Gr}, that there exists $c_1=c_1(n,\Omega)>0$ and $\mu_1=\mu_1(n,\Omega)>0$ such that, for all  $w\in W^{2,q}(\Omega)$, we have
\begin{equation}\label{ae1.1}
\|w\|_{W^{2,q}(\Omega)}\le c_1\left( \|(\Delta -\mu_1)w \|_{L^q(\Omega)}+\|w\|_{W^{2-1/q,q}(\Gamma)}\right).
\end{equation}

\subsection{The difference between two Dirichlet-to-Neumann maps}

Let 
\[
\mathcal{B}=W^{2-1/p,p}(\Gamma)\cap L^{q'}(\Gamma)
\]
 be endowed with the norm
\[
\|\varphi\|:=\|\varphi\|_{W^{2-1/p,p}(\Gamma)}+\|\varphi\|_{L^{q'}(\Gamma)},\quad \varphi \in \mathcal{B}.
\]

In this subsection, $\mathbf{c}=\mathbf{c}(n,\Omega,V_0)>0$ will again denote a generic constant.

Let $V\in \mathscr{V}$, $\varphi\in \mathcal{B}$ and $\lambda\in \Pi_1$. 
Then \eqref{ae1.1} yields
\begin{align*}
\|u_V(\lambda)(\varphi)\|_{W^{2,q}(\Omega)}&\le c_1\|(\Delta -\mu_1)u_V(\lambda)(\varphi)\|_{L^q(\Omega)}+c_1\|\varphi\|
\\
&\le c_1\left(\|Vu_V(\lambda)(\varphi)\|_{L^q(\Omega)}+(|\lambda|+\mu_1)\|u_V(\lambda)(\varphi)\|_{L^q(\Omega)}\right)+c_1\|\varphi\|
\\
&\le \mathbf{c}|\lambda|\|u_V(\lambda)(\varphi)\|_{L^2(\Omega)}+c_1\|\varphi\|,
\end{align*}
where $c_1$ is as in \eqref{ae1.1}. That is we have
\begin{equation}\label{3.3}
\|u_V(\lambda)(\varphi)\|_{W^{2,q}(\Omega)}\le \mathbf{c}|\lambda|\|u_V(\lambda)(\varphi)\|_{L^2(\Omega)}+c_1\|\varphi\|.
\end{equation}

\begin{lemma}\label{lemma1}
Let $V\in \mathscr{V}$ and $\varphi\in \mathcal{B}$. There exists $\mu_\ast=\mu_\ast(n,\Omega,V_0) \ge \mu_0$ such that, for all $\lambda \in \mathbb{C}\setminus \mathbb{R}$ with $-\Re \lambda \ge \mu_\ast$, we have
\begin{align}
 &\|u_V(\lambda)(\varphi)\|_{L^2(\Omega)}\le \mathbf{c}|\Re \lambda|^{-1/2}\| u_V(\lambda)(\varphi)\|_{W^{2,q}(\Omega)}^{1/2}\|\varphi\|^{1/2},\label{3.4}
 \\
 & \|u_V(\lambda)(\varphi)\|_{L^{p'}(\Omega)}\le \mathbf{c}\| u_V(\lambda)(\varphi)\|_{W^{2,q}(\Omega)}^{1/2}\|\varphi\|^{1/2}.\label{3.4.1}
\end{align}
\end{lemma}

\begin{proof}
Let $u:=u_V(\lambda)(\varphi)$. Applying Green's formula, we find
\[
\int_\Omega |\nabla u|^2dx+\int_\Omega (V-\lambda)|u|^2dx=\int_\Gamma \gamma_1u\overline{\varphi}d\sigma .
\]
By proceeding as in the proof of Lemma \ref{lem4.1}, we find $\mu_\ast=\mu_\ast(n,\Omega,V_0) \ge \mu_0$ such that for all $\lambda \in \mathbb{C}\setminus \mathbb{R}$ satisfying $-\Re \lambda \ge \mu_\ast$ we have
\[
\|\nabla u\|_{L^2(\Omega)}^2+|\Re \lambda|\|u\|_{L^2(\Omega)}^2\le \mathbf{c}\|\varphi\gamma_1u\|_{L^1(\Gamma)}
\]
and then
\begin{align*}
\|\nabla u\|_{L^2(\Omega)}^2+|\Re \lambda|\|u\|_{L^2(\Omega)}^2&\le \mathbf{c}\| \gamma_1u\|_{L^q(\Gamma)}\|\varphi\|_{L^{q'}(\Gamma)}
\\
&\le \mathbf{c}\| \gamma_1u\|_{L^q(\Gamma)}\|\varphi\|.
\end{align*}
The proof  is completed by using that  $\gamma_1:W^{2,q}(\Omega)\rightarrow L^q(\Gamma)$ is bounded.
\end{proof}

Define
\[
\Pi_2=\{\lambda \in \mathbb{C}\setminus \mathbb{R};\; -\Re \lambda \ge \max(\lambda_\ast,\mu_\ast),\; |\Im \lambda||\Re \lambda|^{-1}\le 1 \}\; (\subset \Pi_1),
\]
where $\lambda_\ast$ is as in Lemma \ref{lem4.1} and $\mu_\ast$ is as in Lemma \ref{lemma1}.

Let $V\in \mathscr{V}$, $\lambda \in \Pi_2$ and $\varphi\in \mathcal{B}$. By combining \eqref{3.3} and \eqref{3.4}, we derive
\[
\|u_V(\lambda)(\varphi)\|_{W^{2,q}(\Omega)}\le \mathbf{c}\left(|\Re \lambda|^{1/2}\| u_V(\lambda)(\varphi)\|_{W^{2,q}(\Omega)}^{1/2}\|\varphi\|^{1/2}+\|\varphi\|\right).
\]
Thus,
\[
\|u_V(\lambda)(\varphi)\|_{W^{2,q}(\Omega)}\le \frac{1}{2}\| u_V(\lambda)(\varphi)\|_{W^{2,q}(\Omega)}+\mathbf{c}|\Re \lambda|\|\varphi\|
\]
and hence
\begin{equation}\label{3.5}
\|u_V(\lambda)(\varphi)\|_{W^{2,q}(\Omega)}\le \mathbf{c}|\Re \lambda|\|\varphi\|.
\end{equation}
This and \eqref{3.4} imply
\begin{equation}\label{3.6}
 \|u_V(\lambda)(\varphi)\|_{L^2(\Omega)}\le \mathbf{c} \|\varphi\|.
\end{equation}
Similarly, we obtain from  \eqref{3.4.1} and \eqref{3.5}
\begin{equation}\label{3.6.1}
 \|u_V(\lambda)(\varphi)\|_{L^{p'}(\Omega)}\le  \mathbf{c}|\Re \lambda|^{1/2}\|\varphi\|.
\end{equation}

In the remaining part of this subsection, for all  non negative integer $j$,  $\mathbf{c}_j=\mathbf{c}_j(n,\Omega,V_0,j)>0$  will denote  a generic constant.
 
\begin{lemma}\label{lem1.1}
Let $V\in \mathscr{V}$, $\lambda \in \Pi_2$ and $\varphi\in \mathcal{B}$. Then we have, for all $j\ge 0$,
\begin{align}
&\|u_V^{(j)}(\lambda)(\varphi)\|_{L^2(\Omega)}\le \mathbf{c}_j|\Re\lambda|^{-j}\|\varphi\|,\label{eq6.r}
\\
&\|u_V^{(j)}(\lambda)(\varphi)\|_{L^{p'}(\Omega)}\le \mathbf{c}_j|\Re\lambda|^{-j+1/2}\|\varphi\|.\label{eq6.r.1}
\\
&\|u_V^{(j)}(\lambda)(\varphi)\|_{W^{2,q}(\Omega)}\le\mathbf{c}_j|\Re\lambda|^{-j+1}\|\varphi\|, \label{eq5.r}
\end{align}
\end{lemma}

\begin{proof}
Let $u(\lambda):=u_V(\lambda)(\varphi)$. We have from \eqref{d}
\[
(-\Delta +V-\lambda )u^{(1)}(\lambda)=u(\lambda),\quad \gamma_0u^{(1)}=0.
\] 
Thus,
\[
u^{(1)}(\lambda)=R_V(\lambda)(u(\lambda)).
\]
Hence, according to \eqref{3.1},
\begin{equation}\label{eq1}
\|u^{(1)}(\lambda)\|_{W^{2,q}(\Omega)}\le \mathbf{c}\|u(\lambda)\|_{L^2(\Omega)}.
\end{equation}
This and \eqref{3.6}  yield

\begin{equation}\label{eq2}
\|u^{(1)}(\lambda)\|_{W^{2,q}(\Omega)}\le \mathbf{c}\|\varphi\|.
\end{equation}

On the other hand, it follows from \eqref{re1} and \eqref{re2}
\begin{align*}
&\|u^{(1)}(\lambda)\|_{L^2(\Omega)}\le 2|\Re \lambda|^{-1}\|u(\lambda)\|_{L^2(\Omega)},
\\
&\|u^{(1)}(\lambda)\|_{L^{p'}(\Omega)}\le \mathbf{c}|\Re \lambda|^{-1/2}\|u(\lambda)\|_{L^2(\Omega)}.
\end{align*}
By applying \eqref{3.6}, we find
\begin{align}
&\|u^{(1)}(\lambda)\|_{L^2(\Omega)}\le \mathbf{c}|\Re\lambda|^{-1}\|\varphi\|,\label{eq4}
\\
&\|u^{(1)}(\lambda)\|_{L^{p'}(\Omega)}\le \mathbf{c} |\Re\lambda|^{-1/2}\|\varphi\|.\label{eq4.1}
\end{align}
Again, we obtain from \eqref{d}
\[
(-\Delta +V-\lambda )u^{(2)}(\lambda)=2u^{(1)}(\lambda),\quad \gamma_0u^{(2)}=0.
\] 
We repeat the previous calculations replacing $u(\lambda)$ with $u^{(1)}(\lambda)$ in order to obtain
\begin{align*}
&\|u^{(2)}(\lambda)\|_{W^{2,q}(\Omega)}\le\mathbf{c}|\Re\lambda|^{-1}\|\varphi\|, 
\\
&\|u^{(2)}(\lambda)\|_{L^2(\Omega)}\le \mathbf{c} |\Re\lambda|^{-2}\|\varphi\|,
\\
&\|u^{(2)}(\lambda)\|_{L^{p'}(\Omega)}\le \mathbf{c}|\Re\lambda|^{-1/2-1}\|\varphi\|.
\end{align*}

Next, by induction in $j$, $ j\ge 1$, since we have, according to \eqref{d},
\[
(-\Delta +V-\lambda )u^{(j)}(\lambda)=ju^{(j-1)}(\lambda),\quad \gamma_0u^{(j)}=0,
\] 
we argue as above to obtain
\begin{align*}
&\|u^{(j)}(\lambda)\|_{W^{2,q}(\Omega)}\le  \mathbf{c}_j|\Re\lambda|^{-j+1}\|\varphi\|, 
\\
&\|u^{(j)}(\lambda)\|_{L^2(\Omega)}\le  \mathbf{c}_j |\Re\lambda|^{-j}\|\varphi\|,
\\
&\|u^{(j)}(\lambda)\|_{L^{p'}(\Omega)}\le  \mathbf{c}_j |\Re\lambda|^{-j+1/2}\|\varphi\|.
\end{align*}
These inequalities complete the proof.
\end{proof}

We recall that $\varsigma=\frac{n(r-2)}{2r}$. By using an interpolation inequality, we obtain from \eqref{eq6.r} and \eqref{eq6.r.1}
\begin{equation}\label{eq6.r.2}
\|u_V^{(j)}(\lambda)(\varphi)\|_{L^r(\Omega)}\le  \mathbf{c}_j |\Re\lambda|^{-j+\varsigma/2}\|\varphi\|,\quad j\ge 0,
\end{equation}
where $V$, $\lambda$ and $\varphi$ are as in Lemma \ref{lem1.1}.

 Next, we establish the following result. From now on, the constants $\mathbf{c}_j$ ($j\ge 0$) depend also on $W_0$.
 \begin{lemma}\label{lem1.2}
Let $j\ge 0$ be an integer $(V_1,V_2)\in \mathscr{W}$, $\lambda \in \Pi_2$ and $\varphi\in \mathcal{B}$. Then  we have
\begin{align}
&\|u_{V_1}^{(j)}(\lambda)(\varphi)- u_{V_2}^{(j)}(\lambda)(\varphi)\|_{W^{2,q}(\Omega)}\le  \mathbf{c}_j |\Re\lambda|^{-j+\varsigma/2}\|\varphi\|.\label{e7.r}
\\
&\|u_{V_1}^{(j)}(\lambda)(\varphi)- u_{V_2}^{(j)}(\lambda)(\varphi)\|_{L^2(\Omega)}\le \mathbf{c}_j |\Re\lambda|^{-(j+1)+\varsigma/2}\|\varphi\|. \label{e7.1.r}
\end{align}
\end{lemma}
 
 \begin{proof}
 Clearly, $u(\lambda):=u_{V_1}(\lambda)(\varphi)- u_{V_2}(\lambda)(\varphi)\in H_0^1(\Omega)$ satisfies
\begin{equation}\label{e7.1}
(-\Delta +V_1-\lambda)u(\lambda)=(V_2-V_1)u_{V_2}(\lambda)(\varphi).
\end{equation}
Thus, $u(\lambda)=R_{V_1}(\lambda)((V_2-V_1)u_{V_2}(\lambda)(\varphi))$. 
In view of the relation $\frac{1}{2}=\frac{1}{\mathfrak{a}}+\frac{1}{r}$, \eqref{3.1}, combined with an interpolation inequality, yields
\begin{align*}
\|u(\lambda)\|_{W^{2,q}(\Omega)}&\le \mathbf{c}\|(V_2-V_1)u_{V_2}(\lambda)(\varphi)\|_{L^2(\Omega)} 
\\
&\le  \mathbf{c}\|V_2-V_1\|_{L^\mathfrak{a}(\Omega)}\|u_{V_2}(\lambda)(\varphi)\|_{L^r(\Omega)}
\\
&\le \mathbf{c}\|u_{V_2}(\lambda)(\varphi)\|_{L^r(\Omega)}.
\end{align*}

This and \eqref{eq6.r.2} yield
\[
\|u(\lambda)\|_{W^{2,q}(\Omega)} \le \mathbf{c}|\Re \lambda|^{\varsigma/2}\|\varphi\|.
\]

Similarly, we have
\begin{align*}
\|u(\lambda)\|_{L^2(\Omega)}&\le \mathbf{c}|\Re \lambda|^{-1}\|(V_2-V_1)u_{V_2}(\lambda)(\varphi)\|_{L^2(\Omega)}
\\
&\le  \mathbf{c}|\Re \lambda|^{-1}\|u_{V_2}(\lambda)(\varphi)\|_{L^r(\Omega)}
\end{align*}
and then, again by \eqref{eq6.r.2},
\[
\|u(\lambda)\|_{L^2(\Omega)}\le \mathbf{c} |\Re\lambda|^{-1+\varsigma/2}\|\varphi\|.
\]

By taking the derivative with respect to $\lambda$ of each side \eqref{e7.1}, we find
\[
(-\Delta +V_1-\lambda)u^{(1)}(\lambda)=u(\lambda)+(V_2-V_1)u_{V_2}^{(1)}(\lambda)(\varphi).
\]
By proceeding as above, we obtain
\[
\|u^{(1)}(\lambda)\|_{W^{2,q}(\Omega)}\le \mathbf{c} |\Re\lambda|^{-1+\varsigma/2}\|\varphi\|
\]
and 
\[
\|u^{(1)}(\lambda)\|_{L^2(\Omega)}\le  \mathbf{c} |\Re\lambda|^{-2+\varsigma/2}\|\varphi\|.
\]
By using 
\[
(-\Delta +V_1-\lambda)u^{(j)}(\lambda)=ju^{(j-1)}(\lambda)+(V_2-V_1)u_{V_2}^{(j)}(\lambda)(\varphi),
\]
and an induction in $j$, $j\ge 1$, we get similarly as above
\[
\|u^{(j)}(\lambda)\|_{W^{2,q}(\Omega)}\le \mathbf{c}_j |\Re\lambda|^{-j+\varsigma/2}\|\varphi\|
\]
and
\[
\|u^{(j)}(\lambda)\|_{L^2(\Omega)}\le \mathbf{c}_j |\Re\lambda|^{-(j+1)+\varsigma/2}\|\varphi\|.
\]
The proof is then complete.
\end{proof}

The following interpolation inequality is contained in \cite[Theorem 1.4.3.3]{Gr}:
\begin{equation}\label{ii}
\|w\|_{W^{t,q}(\Omega)}\le c\|w\|_{W^{2,q}(\Omega)}^{t/2}\|w\|_{L^q(\Omega)}^{1-t/2},
\end{equation}
where $c=c(n,\Omega,t)>0$ is a constant. Under the assumptions of Lemma \ref{lem1.2}, by using \eqref{e7.r} and \eqref{e7.1.r} in \eqref{ii} with $w=u_{V_1}^{(j)}(\lambda)(\varphi)- u_{V_2}^{(j)}(\lambda)(\varphi)$, we get
\[
\|u_{V_1}^{(j)}(\lambda)(\varphi)- u_{V_2}^{(j)}(\lambda)(\varphi)\|_{W^{t,q}(\Omega)}\le  \mathbf{c}_j |\Re\lambda|^{-(j+1-t/2-\varsigma/2)}\|\varphi\|,
\]
Given that the trace map $w\in W^{t,q}(\Omega)\mapsto \gamma_1w\in L^q$ is bounded (e.g. \cite[Theorem 1.5.1.2]{Gr}), we obtain
\[
\|\Lambda_{V_1}^{(j)}(\lambda)(\varphi)- \Lambda_{V_2}^{(j)}(\lambda)(\varphi)\|_{L^q(\Gamma)}\le  \mathbf{c}_j |\Re\lambda|^{-(j+1-t/2-\varsigma/2)}\|\varphi\|
\]
or, equivalently,
\begin{equation}\label{reso1}
\|\Lambda_{V_1}^{(j)}(\lambda)(\varphi)- \Lambda_{V_2}^{(j)}(\lambda)(\varphi)\|_{L^q(\Gamma)}\le  \mathbf{c}_j |\Re\lambda|^{-j-\upsilon}\|\varphi\|.
\end{equation}

Define
\begin{equation}\label{pi+}
\Pi=\left\{\lambda =-\tau^\sigma +2i\tau\in \Pi_2 \right\} .
\end{equation}

In view of \eqref{reso1}, we can state the following corollary.
 \begin{corollary}\label{corDN1}
Let $j\ge 0$ be an integer, $(V_1,V_2)\in \mathscr{W}$, $\lambda=-\tau^\sigma +2i\tau \in  \Pi$ and $\varphi \in \mathcal{B}$. Then we have
\begin{equation}\label{DN1}
\|\Lambda_{V_1}^{(j)}(\lambda)(\varphi)- \Lambda_{V_2}^{(j)}(\lambda)(\varphi)\|_{L^q(\Gamma)}\le  \mathbf{\tilde{c}}\tau^{-3-j\sigma}\|\varphi\|,
\end{equation}
where $\mathbf{\tilde{c}}=\mathbf{\tilde{c}}(n,\Omega,V_0,W_0,\varepsilon,r)>0$ is a constant.
\end{corollary}

\section{Proof of Theorem \ref{mainthm1}}

We recall, for convenience, that for all $(V_1,V_2)\in \mathscr{W}$, we have 
\[
\mathscr{D}(V_1,V_2):=\sum_{k\ge 1}k^{-2\ell/n}\left[ |\lambda_{V_1}^k-\lambda_{V_2}^k|+\|\psi_{V_2}^k-\psi_{V_1}^k\|_{L^q(\Gamma)}\right] .
\]
We also recall that $\mathcal{B}:=W^{2-1/p,p}(\Gamma)\cap L^{q'}(\Gamma)$ is equipped with the norm
\[
\|\varphi\|=\|\varphi\|_{W^{2-1/p,p}(\Gamma)}+\|\varphi\|_{L^{q'}(\Gamma)},\quad \varphi \in \mathcal{B}.
\]

In the remaining part of this section, $\mathbf{c}=\mathbf{c}(n,\Omega,V_0,W_0,\varepsilon,r,\ell)>0$ will denote a generic constant. 

Before proceeding to the proof of Theorem \ref{mainthm1}, we will first prove the following intermediate result. Define
\[
\Sigma_{\tau^\ast}=\{\lambda_\tau=(\tau+i)^2\in \Sigma;\; \tau\ge \tau^\ast\},\quad \tau^\ast \ge 2.
\]

\begin{proposition}\label{promt1}
Let $(V_1,V_2)\in \mathscr{W}$, $\varphi\in \mathcal{B}$. There exists $\tau^\ast=\tau^\ast (n,\Omega,V_0)$ such that, for all $\lambda_\tau=(\tau+i)^2\in \Sigma_{\tau^\ast}$, we have
\begin{equation}\label{2.13}
 \|\Lambda_{V_1}(\lambda_\tau)(\varphi)-\Lambda_{V_2}(\lambda_\tau)(\varphi)\|_{L^q(\Gamma)}\le \mathbf{c}\left(\tau^{-3} 
+\tau^{\sigma(\ell+1) +n+2}\mathscr{D}(V_1,V_2)\right)\|\varphi\|. 
\end{equation}
\end{proposition}

\begin{proof}
Define for  $s\in [0,1]$ and $\tau \ge 2$
\[
\mu_\tau^s=-\tau^\sigma+s[\tau^\sigma+\tau^2-1]+2i\tau
\] 
in such a way that $\mu_\tau^0=-\tau^\sigma+2i\tau\in \Pi$, where $\Pi$ is given by \eqref{pi+}, and $\mu_\tau^1=\lambda_\tau=(\tau+i)^2\in \Sigma$.

Let $V_1,V_2\in \mathscr{V}$. By applying Taylor's formula, we get for $j=1,2$
 \begin{align}\label{2.4}
& u_{V_j}(\lambda_\tau)(\varphi)=u_{V_j}(\mu_\tau^0)(\varphi)+\sum_{k=1}^{\ell-1}\frac{1}{k!}(\tau^\sigma+\tau^2-1)^ku_{V_j}^{(k)}(\mu_\tau^0)(\varphi)
 \\
 &\hskip 3cm+\frac{1}{(\ell-1)!}(\tau^\sigma+\tau^2-1)^\ell\int_0^1(1-s)^\ell u_{V_j}^{(\ell)}(\mu_\tau^s)(\varphi)ds. \nonumber
 \end{align} 
 Hence, \eqref{DN1} implies
 \begin{align}
& \mathbf{c} \| \Lambda_{V_1}(\lambda_\tau)(\varphi)- \Lambda_{V_2}(\lambda_\tau)(\varphi)\|_{L^q(\Gamma)}\le \tau^{-3}\|\varphi\| \label{xx}
 \\
&\hskip 3cm  +\tau^{\ell \sigma} \int_0^1\| \Lambda_{V_1}^{(\ell)}(\mu_\tau^s)(\varphi)- \Lambda_{V_2}^{(\ell)}(\mu_\tau^s)(\varphi)\|_{L^q(\Gamma)}ds.\nonumber
 \end{align}
 
 In light of \eqref{1.6}, we find  an integer $k_\ast=k_\ast(n,\Omega,V_0)>0$ and a constant $\mathbf{c}_0=(n,\Omega,V_0)>0$ such that
\begin{equation}\label{x1}
\lambda_V^k\ge \mathbf{c}_0k^{2/n},\quad V\in \mathscr{V},\; k\ge k_\ast.
\end{equation}
On the other hand, we have
\begin{equation}\label{x2}
-\Re \mu_\tau^s=\tau^\sigma -s[\tau^\sigma +\tau^2-1]\ge -\tau^2, \quad s\in [0,1].
\end{equation}
Let $k_\tau$ be the smallest integer such that $k_\tau^{2/n}\ge 2\mathbf{c}_0^{-1}\tau^2$. We then choose $\tau^\ast=\tau^\ast(n,\Omega, V_0) \ge 2$ such that $k_\tau \ge k_\ast$, for all $\tau \ge \tau^\ast$. Then, by combining \eqref{x1} and \eqref{x2}, we obtain
\begin{equation}\label{x3}
\lambda_V^k/2-\Re \mu_\tau^s\ge 0,\quad s\in [0,1],\; V\in \mathscr{V},\; k\ge k_\tau.
\end{equation}

For $V=V_1$ or $V=V_2$ and $\lambda \in \mathbb{C}\setminus\mathbb{R}$,  we use the following decomposition $\gamma_1u_V^{(\ell)}(\lambda)(\varphi)=\gamma_1v_V^{(\ell)}(\lambda)(\varphi)+\gamma_1w_V^{(\ell)}(\lambda)(\varphi)$, where
\[
\gamma_1w_V(\lambda)(\varphi)=-\ell !\sum_{k< k_\tau}\frac{\langle \varphi, \psi_V^k\rangle}{(\lambda_V^k-\lambda)^{\ell+1}}\psi_V^k,\quad \gamma_2v_V^{(\ell)}(\lambda)(\varphi)= -\ell !\sum_{k\ge k_\tau}\frac{\langle \varphi, \psi_V^k\rangle}{(\lambda_V^k-\lambda)^{\ell+1}}\psi_V^k.
\]
We split
\begin{align*}
&\gamma_1v_{V_1}^{(\ell)}(\mu_\tau^s)(\varphi)-\gamma_1v_{V_2}^{(\ell)}(\mu_\tau^s)(\varphi)= 
\\
&\hskip 2.5cm \ell !\sum_{k\ge k_\tau}\frac{\langle \varphi, \psi_{V_2}^k\rangle}{(\lambda_{V_2}^k-\mu_\tau^s)^{\ell +1}}\phi_{V_2}^k-\ell !\sum_{k\ge k_\tau }\frac{\langle \varphi, \psi_{V_1}^k\rangle}{(\lambda_{V_1}^k-\mu_\tau^s)^{\ell +1}}\psi_{V_1}^k.
\end{align*}
into three terms:
\begin{equation}\label{2.5}
\gamma_1v_{V_1}^{(1)}(\mu_\tau^s)(\varphi)-\gamma_1v_{V_2}^{(1)}(\mu_\tau^s)(\varphi)=\sum_{j=1}^3I_j,
\end{equation}
where
\begin{align*}
&I_1=\ell !\sum_{k\ge k_\tau}\left(\frac{1}{(\lambda_{V_2}^k-\mu_\tau^s)^{\ell +1}}-\frac{1}{(\lambda_{V_1}^k-\mu_\tau^s)^{\ell +1}}\right)\langle \varphi, \psi_{V_2}^k\rangle\psi_{V_2}^k,
\\
&I_2=\ell !\sum_{k\ge  k_\tau}\frac{\langle \varphi, \psi_{V_1}^k-\psi_{V_2}^k\rangle}{(\lambda_{V_1}^k-\mu_\tau^s)^{\ell +1}}\psi_{V_2}^k,
\\
&I_3=\ell !\sum_{k\ge k_\tau}\frac{\langle \varphi, \psi_{V_1}^k\rangle}{(\lambda_{V_1}^k-\mu_\tau^s)^{\ell +1}}(\psi_{V_2}^k-\psi_{V_1}^k).
\end{align*}
Let $k\ge k_\tau$. Using that
\begin{align*}
&\frac{1}{(\lambda_{V_2}^k-\mu_\tau^s)^{\ell +1}}-\frac{1}{(\lambda_{V_1}^k-\mu_\tau^s)^{\ell +1}}
\\
&\hskip 2cm = \frac{(\lambda_{V_1}^k-\lambda_{V_2}^k)\sum_{j=0}^{\ell}(\lambda_{V_1}^k-\mu_\tau^s)^{\ell-j} (\lambda_{V_2}^k-\mu_\tau^s)^{j}} {(\lambda_{V_1}^k-\mu_\tau^s)^{\ell +1}(\lambda_{V_2}^k-\mu_\tau^s)^{\ell +1}}
\\
&\hskip 2cm =(\lambda_{V_1}^k-\lambda_{V_2}^k)\sum_{j=0}^{\ell}(\lambda_{V_1}^k-\mu_\tau^s)^{-1-j} (\lambda_{V_2}^k-\mu_\tau^s)^{-\ell -1 +j},
\end{align*}
we obtain
\[
\left|\frac{1}{(\lambda_{V_2}^k-\mu_\tau^s)^{\ell +1}}-\frac{1}{(\lambda_{V_1}^k-\mu_\tau^s)^{\ell +1}}\right| \le\mathbf{c} k^{-2(\ell +2)/n} |\lambda_{V_1}^k-\lambda_{V_2}^k|.
\] 
This, combined with
\begin{equation}\label{eee}
 \|\psi_{V_2}^k\|_{L^q(\Gamma)}\le \mathbf{c}(1+|\lambda_{V_2}^k|),\quad k\ge 1,
\end{equation}
implies 
\begin{equation}\label{2.11.1}
\| I_1\|_{L^2(\Gamma)}\le \mathbf{c}\sum_{k\ge k_\tau} k^{-2\ell /n} |\lambda_{V_1}^k-\lambda_{V_2}^k|\|\varphi\|.
\end{equation}
We proceed similarly as above to obtain
\begin{equation}\label{2.11.3}
\|I_2\|_{L^q(\Gamma)}+\|I_3\|_{L^q(\Gamma)}\le \sum_{k\ge k_\tau} k^{-2\ell/n}\|\psi_{V_2}^k-\psi_{V_1}^k\|_{L^p(\Gamma)}\|\varphi\|.
\end{equation}
Combining  \eqref{2.11.1} and \eqref{2.11.3}, we find
\begin{equation}\label{2.12.0}
\|\gamma_1v_{V_1}^{(1)}(\mu_\tau^s)(\varphi)-\gamma_1v_{V_2}^{(1)}(\mu_\tau^s)(\varphi)\|_{L^q(\Gamma)}\le \mathbf{c}\mathscr{D}(V_1,V_2)\|\varphi\|.
\end{equation}
In view of the definition of $k_\tau$, we verify that 
\begin{equation}\label{2.12.1}
\|\gamma_1w_{V_1}^{(\ell)}(\mu_\tau^s)(\varphi)-\gamma_1w_{V_2}^{(\ell)}(\mu_\tau^s)(\varphi)\|_{L^q(\Gamma)}\le \mathbf{c}\tau^{n+\ell+2}\mathscr{D}(V_1,V_2)\|\varphi\|.
\end{equation}
Putting \eqref{2.12.0} and \eqref{2.12.1} in \eqref{xx}, we obtain
\[
 \| \Lambda_{V_1}(\lambda_\tau)- \Lambda_{V_2}(\lambda_\tau)\|_{L^q(\Gamma)}\le \mathbf{c}\left( \tau^{-3}+\tau^{\sigma(\ell+1) +n+2}\mathscr{D}(V_1,V_2)\right)\|\varphi\|.
\]
This is the expected inequality.
\end{proof}

We are now ready to complete the proof of Theorem \ref{mainthm1}. To this end, let $\varphi_{\lambda ,\omega}(x)=e^{i\sqrt{\lambda}\omega \cdot x}$, with $\lambda \in \mathbb{C}\setminus (-\infty ,0]$ (choosing the standard branch of the square root) and $\omega \in \mathbb{S}^{n-1}$. Define 
\[
S_V(\lambda ,\omega ,\theta )=\int_\Gamma \Lambda_V (\lambda )(\varphi_{\lambda ,\omega})\varphi_{\lambda ,-\theta}d\sigma ,\quad \lambda \in \rho(A_V)\setminus (-\infty ,0],\; \omega ,\; \theta \in \mathbb{S}^{n-1}.
\]
 Following \cite[Lemma 2.2]{Isozaki89}, we have for all $\lambda \in \rho(A_V)\setminus (-\infty ,0]$ and $\omega ,\; \theta \in \mathbb{S}^{n-1}$
 \begin{align}
 &S_V(\lambda ,\omega ,\theta ) = -\frac{\lambda}{2}|\theta -\omega |^2\int_\Omega e^{-i\sqrt{\lambda}(\theta -\omega )\cdot x}dx \label{1.22}
 \\
&\hskip 3cm +\int_\Omega e^{-i\sqrt{\lambda}(\theta -\omega )\cdot x}Vdx-\int_\Omega R_V(\lambda )(V\varphi_{\lambda ,\omega})V\varphi_{\lambda ,-\theta}dx. \nonumber
 \end{align}

 We fix $\xi \in \mathbb{R}^n$ and $\eta \in \mathbb{S}^{n-1}$ such that $\eta \bot \xi$. For $\tau >1$, let
 \[
 \left\{
 \begin{array}{lll}
 &\theta _\tau = c_\tau \eta +(2\tau)^{-1}\xi,
 \\
& \omega _\tau = c_\tau \eta -(2\tau)^{-1}\xi,
\\
 &\sqrt{\lambda _\tau}=\tau +i .
 \end{array}
 \right.
 \]
 
 Let $(V_1,V_2)\in \mathscr{W}$. Then set $V=(V_1-V_2)\chi_\Omega$ and assume that $\lambda_\tau=(\tau+i)^2\in \Sigma_{\tau^\ast}$, where $\Sigma_{\tau^\ast}$ is as in Proposition \ref{promt1}. Then \eqref{1.22} yields
 \begin{align}
 &|\widehat{V}((1+i/\tau)\xi)| 
\le |S_{V_1}(\lambda_\tau ,\omega_\tau ,\theta_\tau )-S_{V_2}(\lambda_\tau ,\omega_\tau ,\theta_\tau )|\label{1.23}
\\
&\hskip 3cm+\left| \int_\Omega R_{V_1}(\lambda_\tau )(V_1\varphi_{\lambda_\tau ,\omega_\tau})V_1\varphi_{\lambda_\tau ,-\theta_\tau}dx\right|\nonumber
\\
& \hskip 4cm +\left| \int_\Omega R_{V_2}(\lambda_\tau )(V_2\varphi_{\lambda_\tau ,\omega_\tau})V_2\varphi_{\lambda_\tau ,-\theta_\tau}dx\right|.\nonumber
 \end{align}
 Here $\widehat{V}$ denotes the Laplace-Fourier transform of $V$.
  
 Let $\rho=p_{1/4}=4n/(2n+1)$ and $\rho'=p'_{1/4}=4n/(2n-1)$ ($p_\theta$ and $p'_\theta$ are defined in \eqref{theta}). For $W=V_1, V_2$ and $\zeta=\omega_\tau, -\theta_\tau$,  H\"older's inequality yields
 \[
 \|W\varphi_{\lambda_\tau,\zeta}\|_{L^\rho(\Omega )}\le \|W\|_{L^{n/2}(\Omega)}\|\varphi_{\lambda_\tau,\zeta}\|_{L^{\tilde{q}}(\Omega)},
 \]
where $\tilde{q}=4n/(2n-7)$. Thus, we have
 \[
 \|W\varphi_{\lambda_\tau,\zeta}\|_{L^\rho(\Omega )}\le \mathbf{c},
 \]
which, in combination with \eqref{1.21}, gives 
\begin{equation}\label{1.24}
\left| \int_\Omega R_{V_j}(\lambda )(V_j\varphi_{\lambda_\tau ,\omega_\tau})V_j\varphi_{\lambda_\tau ,-\theta_\tau}dx\right|\le \mathbf{c}\tau^{-1/2},\quad j=1,2. 
\end{equation}
By using \eqref{1.24} in \eqref{1.23}, we obtain
\begin{equation}\label{1.25}
|\widehat{V}((1+i/\tau)\xi)|\le |S_{V_1}(\lambda_\tau ,\omega_\tau ,\theta_\tau )-S_{V_2}(\lambda_\tau ,\omega_\tau ,\theta_\tau )|+\mathbf{c}\tau^{-1/2}.
\end{equation}
According to the mean-value theorem, we have
\[
|\widehat{V}(\xi)|\le |\widehat{V}((1+i/\tau)\xi)|+\tau^{-1}\sup_{0\le s\le 1}|\nabla \widehat{V}((1+is/\tau)\xi)\cdot \xi|.
\]
On the other hand, we have for $1\le k\le n$
\[
|\partial_k \widehat{V}((1+i/\tau)\xi)((1+is/\tau)\xi)|=|\widehat{x_kV}((1+is/\tau)\xi)|\le e^{\mathbf{c}\tau^{-1}|\xi|}.
\]
Thus,
\[
|\widehat{V}(\xi)|\le |\hat{V}((1+i/\tau)\xi)|+\mathbf{c}\tau^{-1}|\xi|e^{\mathbf{c}\tau^{-1}|\xi|}.
\]
This in \eqref{1.25} yields
\begin{equation}\label{1.26}
|\widehat{V}(\xi)|\le |S_{V_1}(\lambda_\tau ,\omega_\tau ,\theta_\tau )-S_{V_2}(\lambda_\tau ,\omega_\tau ,\theta_\tau )|+\mathbf{c}\tau^{-1/2}+ \mathbf{c}\tau^{-1}|\xi|e^{\mathbf{c}\tau^{-1}|\xi|}.
\end{equation}

Also, we have
\begin{align*}
&|S_{V_1}(\lambda_\tau ,\omega_\tau ,\theta_\tau )-S_{V_2}(\lambda_\tau ,\omega_\tau ,\theta_\tau )|
\\
&\hskip 2cm \le \|\Lambda_{V_1}(\lambda_\tau)(\varphi_{\lambda_\tau ,\omega_\tau})-\Lambda_{V_2}(\lambda_\tau)(\varphi_{\lambda_\tau ,\omega_\tau})\|_{L^q(\Gamma)}\|\varphi_{\lambda_\tau ,-\theta_\tau}\|_{L^{q'}(\Gamma)}.
\\
&\hskip 2cm \le \mathbf{c} \|\Lambda_{V_1}(\lambda_\tau)(\varphi_{\lambda_\tau ,\omega_\tau})-\Lambda_{V_2}(\lambda_\tau)(\varphi_{\lambda_\tau ,\omega_\tau})\|_{L^q(\Gamma)},
\end{align*}
which, in light of \eqref{2.13}, gives
\[
\mathbf{c}|S_{V_1}(\lambda_\tau ,\omega_\tau ,\theta_\tau )-S_{V_2}(\lambda_\tau ,\omega_\tau ,\theta_\tau )|\le \tau^{-1/2} +\tau^{\sigma(\ell +1)+n+4}\mathscr{D}(V_1,V_2),
\]
where we used
\[
\|\varphi_{\lambda_\tau ,\omega}\|\le \mathbf{c}\tau^2.
\]
Therefore, we obtain from \eqref{1.26}
\begin{align}
&\int_{|\xi|\le \tau^\alpha}(1+|\xi|^2)^{-1}|\widehat{V}(\xi)|^2d\xi \le \mathbf{c}\tau^{2[\sigma(\ell +1)+n+4]+n\alpha}\mathscr{D}(V_1,V_2)^2\label{1.27}
\\
&\hskip 5.5cm +\mathbf{c}\tau^{-1+n\alpha}+\mathbf{c}\tau^{-2+(n+2)\alpha}e^{\mathbf{c}\tau^{-1+\alpha}},\nonumber
\end{align}
where $\alpha > 0$ will be chosen later.

We verify that 
\begin{equation}\label{1.28}
\int_{|\xi|> \tau^\alpha}(1+|\xi|^2)^{-1}|\widehat{V}(\xi)|^2d\xi \le \tau^{-2\alpha}\int_{\mathbb{R}^n} |\widehat{V}(\xi)|^2dx.
\end{equation}
According to Parseval's identity, we have
\[
\int_{\mathbb{R}^n} |\widehat{V}(\xi)|^2dx=\int_{\mathbb{R}^n} |V(\xi)|^2dx=\|V_1-V_2\|_{L^2(\Omega)}\le \mathbf{c}
\]
and then \eqref{1.28} implies
\begin{equation}\label{1.29}
\int_{|\xi|> \tau^\alpha}(1+|\xi|^2)^{-1}|\widehat{V}(\xi)|^2d\xi \le \mathbf{c} \tau^{-2\alpha}.
\end{equation}
By putting together \eqref{1.27} and \eqref{1.29}, we get
\begin{align}
&\|V\|_{H^{-1}(\mathbb{R}^n)}^2 \le  \mathbf{c}\tau^{2[\sigma(\ell +1)+n+4]+n\alpha}\mathscr{D}(V_1,V_2)^2\label{1.30}
\\
&\hskip 4.5cm +\mathbf{c}\tau^{-1+n\alpha}+\mathbf{c}\tau^{-2+(n+2)\alpha}e^{\mathbf{c}\tau^{-1+\alpha}}+\mathbf{c} \tau^{-2\alpha}.\nonumber
\end{align}
If $0<\alpha <1$, then \eqref{1.30} implies
\begin{align}
&\|V\|_{H^{-1}(\mathbb{R}^n)}^2 \le   \mathbf{c}\tau^{2[\sigma(\ell +1)+n+4]+n\alpha}\mathscr{D}(V_1,V_2)^2\label{1.31}
\\
&\hskip 4.5cm +\mathbf{c}\tau^{-1+n\alpha}+\mathbf{c}\tau^{-2+(n+2)\alpha} +\mathbf{c}\tau^{-2\alpha}.\nonumber
\end{align}
By taking $\alpha=1/(n+2)$ in \eqref{1.31}, we obtain
\[
\|V\|_{H^{-1}(\mathbb{R}^n)} \le  \mathbf{c}\tau^{\sigma(\ell +1)+n+4+n/(2n+4)}\mathscr{D}(V_1,V_2)+\mathbf{c}\tau^{-1/(n+2)}.
\]
This and the fact that $\|V_1-V_2\|_{H^{-1}(\Omega)}\le \|V\|_{H^{-1}(\mathbb{R}^n)}$ give
\begin{equation}\label{1.32}
\|V_1-V_2\|_{H^{-1}(\Omega)} \le  \mathbf{c}\tau^{\sigma(\ell +1)+n+4+n/(2n+4)}\mathscr{D}(V_1,V_2)+\mathbf{c}\tau^{-1/(n+2)}.
\end{equation}
Finally, by minimizing with respect to $\tau$, we obtain \eqref{mt1} from  \eqref{1.32}.

\section{Extension to the anisotropic case}

In this section, we explain how to modify the previous proof to obtain an extension of Theorem \ref{mainthm1} to the anisotropic case. We adopt the definitions and notations from the preceding sections.

Suppose that $\Omega$ is of class $C^\infty$. Let $\mathbf{g}=(g_{k\ell})$ be a Riemannian metric of class $C^\infty$ on $\overline{\Omega}$. We assume that $\mathbf{g}$ is simple in $\Omega$, meaning that $\Gamma$ is strictly convex with respect to the metric $\mathbf{g}$ and that, for every $x\in \overline{\Omega}$, the exponential map $\exp_x:\exp_x^{-1}(\overline{\Omega})\rightarrow \overline \Omega$ is a diffeomorphism. Note that $\mathbf{g}$ can be extended to a simple metric on a domain $\Omega_1$ such that $\Omega \Subset \Omega_1$.

All results from the preceding sections, up to Proposition \ref{promt1}, remain valid. However, we must replace the functions $\varphi_{\lambda,\omega}$ with geometric optics solutions associated with the metric $\mathbf{g}$. Before going into the details, we introduce a few definitions and results that we will use in this section. For all $x \in \Omega_1$, recall that the exponential map $\mathrm{exp}_x$ is defined as follows:
\[
\mathrm{exp}_x: v\in T_x\Omega_1\mapsto \gamma_{x,\theta}(|v|)\in \Omega_1,\quad \theta =\frac{v}{|v|},
\]
where $\gamma_{x,\theta}$ is the unique geodesic satisfying the initial conditions $\gamma_{x,\theta}(0)=x$ and $\dot{\gamma}_{x,\theta}(0)=v$.

For the sake of simplicity, the inner products on $T_x\Omega_1$ and $T^\ast_x\Omega_1$ ($x\in \Omega_1$) will be denoted by the same symbol $\langle \cdot,\cdot \rangle$. The following notations will be used subsequently.
\begin{align*}
&S\Omega_1:=\{ (x,\theta)\in T\Omega_1;\; |\theta|=1\},
\\
&\partial S\Omega_1:=\{(x,\theta)\in S\Omega_1,\; x\in \partial \Omega_1\},
\\
&\partial_\pm S\Omega_1:=\{ (x,\theta)\in S\Omega_1;\; x\in \partial \Omega_1,\; \pm \langle \theta|\nu(x)\rangle<0\},
\end{align*}
where $\nu$ denotes the unit outer normal vector field on $\partial \Omega_1$.

Since $\Omega_1$ is simple, for every $(x,\theta)\in S\Omega_1$, $\mathrm{exp}_x$ is defined on a maximal finite interval $[\ell_-(x,\theta),\ell_+(x,\theta)]$ such that $\ell_\pm(x,\theta)\in \partial \Omega_1$.

Also, for $ x\in \Omega_1$, set
\begin{align*}
&S_x\Omega_1:=\{\theta\in T_x\Omega_1;\; |\theta|=1\},
\\
&S^+_x\Omega_1:=\{\theta \in S_xN;\; \langle \theta|\nu(x)\rangle >0\},
\end{align*}
and let $d\omega_x(\theta)$ be the Riemannian measure on $S_x\Omega_1$, which induces the Riemannian measures $d\mu(x) d\omega_x(\theta)$ and $ds(x) d\omega_x(\theta)$ on $S\Omega_1$ and $\partial S\Omega_1$, respectively.

The following Santal\'o formula will be used subsequently:
\begin{align*}
&\int_{S\Omega_1}F(x,\theta)d\mu(x)d\omega_x(\theta)
\\
&\hskip 2cm =\int_{\partial_+S\Omega_1}\left(\int_0^{\ell_+(x,\theta)}F(\gamma_{x,\theta}(t),\dot{\gamma}_{x,\theta}(t))dt\right) |\langle\theta|\nu(x)\rangle| ds(x) d\omega_x(\theta).
\end{align*}

\subsection{Geodesic ray transform} The geodesic ray transform for functions is given as follows:
\[
\mathbf{I}:C^\infty(\Omega_1)\rightarrow C^\infty(\partial_+S\Omega_1): f\mapsto \mathbf{I}(f)(x,\theta):=\int_0^{\ell_+(x,\theta)}f(\gamma_{x,\theta}(t))dt.
\]
For all integer $k \ge 0$, $\mathbf{I}$ extends to a bounded operator, still denoted by $\mathbf{I}$, from $H^k(\Omega_1)$ to $H^k(\partial_+S\Omega_1)$. Furthermore, $\mathbf{I} : L^2(\Omega_1) \to L^2(\partial_+S\Omega_1)$ is injective.

In the following, $L^2(\partial_+S\Omega_1, |\langle\theta|\nu(x)\rangle| ds(x) d\omega_x(\theta))$ will be denoted by $L_w^2(\partial_+S\Omega_1)$. It is well known (e.g. \cite{Sh}) that the adjoint of $\mathbf{I}:L^2(\Omega_1)\mapsto L_w^2(\partial_+S\Omega_1)$ is given as follows:

\[
\mathbf{I}^\ast: L_w^2(\partial_+S\Omega_1)\rightarrow L^2(\Omega_1): g\mapsto \mathbf{I}^\ast(g)(x)=\int_{S_x\Omega_1}\tilde{g}(x,\theta)d\omega_x(\theta),
\]
where $\tilde{g}$ is the extension of $g$ from $\partial_+S\Omega_1$ to $S\Omega_1$ that is constant along each orbit of the geodesic flow, that is to say
\[
\tilde{g}(x,\theta)=g(\gamma_{x,\theta}(\ell_-(x,\theta))).
\]
The operator $\mathbf{N}:=\mathbf{I}^\ast\mathbf{I}: L^2(\Omega_1)\mapsto L^2(\Omega_1)$ is usually called the normal operator.

The following double inequality will be used subsequently: there exist two constants $c_0>0$ and $c_1>0$ such that
 \begin{equation}\label{nor1}
 c_0\|f\|_{L^2(\Omega)}\le \|\mathbf{N}(f)\|_{H^1(\Omega_1)}\le c_1\|f\|_{L^2(\Omega)},\quad f\in L^2(\Omega),
 \end{equation}
where $L^2(\Omega)$ is viewed as a subspace of $L^2(\Omega_1)$ via the embedding $f\in L^2(\Omega)\mapsto \chi_\Omega f\in L^2(\Omega_1)$, where $\chi_\Omega$ denotes the characteristic function of $\Omega$. Furthermore, for any integer $j\ge 1$ and any open set $\mathcal{O} \subset \Omega_1$, there exists a constant $c_j>0$ such that for all $f\in H^j(\Omega_1)$ with $\mathrm{supp}(f)\subset \mathcal{O}$, we have
 \begin{equation}\label{nor2}
\|\mathbf{N}(f)\|_{H^{j+1}(\Omega_1)}\le c_j\|f\|_{H^j(\Omega)}.
\end{equation}

\subsection{Special solutions} Let $\tilde{x}\in \partial \Omega_1$ be arbitrarily fixed, and let $(r,\theta)$ denote the normal polar coordinates centered at $\tilde{x}$. For simplicity, for $x=\mathrm{exp}_{\tilde{x}}(r\theta)$, $r>0$, and $\theta\in S_{\tilde{x}}\Omega_1$, we use $w(x)$ or $w(r,\theta)$ interchangeably. With this convention, we have
\[
\mathbf{g}(r,\theta)=dr^2+\mathbf{g}_0(r,\theta).
\]

If $\mathbf{d}$ denotes the geodesic distance, then $\phi(x)=\mathbf{d}(x,\tilde{x})=r$ belongs to $C^\infty (\overline{\Omega})$ and it is the solution of the eikonal equation
\begin{equation}\label{eikonal} 
|\nabla \phi|^2:=\sum_{k,\ell=1}^ng^{k\ell}\frac{\partial \phi}{\partial x_k}\frac{\partial \phi}{ \partial x_\ell}=1.
\end{equation}

Let $\rho(r,\theta)=\mathrm{det}(\mathbf{g}_0(r,\theta))$. If $\eta=\eta(\tilde{x},\cdot) \in H^3(S^+_{\tilde{x}}\Omega_1)$  and $a =\rho^{-\frac{1}{4}}\eta \in H^3(\Omega)$, then we verify that $a$ is the solution of the transport equation
\begin{equation}\label{transport}
\sum_{k,\ell=1}^ng^{k\ell}\frac{\partial \phi}{\partial x_k} \frac{\partial a}{\partial x_\ell}+\frac{1}{2}(\Delta \phi) a=0.
\end{equation}
A solution of the \eqref{eikonal} is generally called a phase and a solution of \eqref{transport} an amplitude.

Let $\tau\ge 1$ and $\lambda =(\tau+i)^2$. Then define 
\[
\varphi_+=\varphi_+(a):=e^{i\sqrt{\lambda}\phi}a.
\]
We verify that $\varphi_+$ is the solution of the equation
\begin{equation}\label{u5}
(\Delta+\lambda)\varphi_+= e^{i\sqrt{\lambda}\phi}(\Delta a).
\end{equation}
We define similarly
\[
\varphi_-=\varphi_-(a):=e^{i\overline{\sqrt{\lambda}}\phi}a.
\]

Unless otherwise specified, $\phi$ will denote a smooth solution of the eikonal equation \eqref{eikonal}, and $a, b \in H^3(\Omega)$ will denote two solutions of the transport equation \eqref{transport}.

Let $A_0$ be the unbounded operator acting on $L^2(\Omega)$ as follows:
\[
A_0u=-\Delta u,\quad u\in D(A_0):=H^2(\Omega)\cap H_0^1(\Omega).
\]

For $\zeta\in \{a,b\}$, let 
\begin{align*}
&\varphi^+_{\zeta,\tau} :=e^{ i\sqrt{\lambda}\phi }\zeta+(A_0-\lambda)^{-1}(e^{ i\sqrt{\lambda}\phi }\Delta \zeta),
\\
&\varphi^-_{\zeta,\tau} :=e^{i\overline{\sqrt{\lambda}}\phi }\zeta+(A_0-\overline{\lambda})^{-1}(e^{i\overline{\sqrt{\lambda}}\phi }\Delta \zeta).
\end{align*}
In view of \eqref{u5}, we have $\varphi^{\pm}_{\zeta,\tau}\in H^3(\Omega)$. In addition, we verify that $(\Delta +\lambda)\varphi^+_{\zeta,\tau}=0$ and $(\Delta +\overline{\lambda})\varphi^-_{\zeta,\tau}=0$.

In the following, $\mathbf{c}_0=\mathbf{c}_0(n,\Omega)>0$ will denote a generic constant.

If $\lambda \in \Sigma \subset \Pi_0$, then we have from \eqref{1.14.0}
\begin{align*}
&\|(A_0-\lambda)^{-1}(e^{ i\sqrt{\lambda}\phi }\Delta \zeta)\|_{L^{p'}(\Omega)}\le \mathbf{c}_0\|\zeta\|_{H^2(\Omega)},
\\
&\|(A_0-\overline{\lambda})^{-1}(e^{i\overline{\sqrt{\lambda}}\phi }\Delta \zeta)\|_{L^{p'}(\Omega)}\le \mathbf{c}_0\|\zeta\|_{H^2(\Omega)}.
\end{align*}
Thus,
\begin{equation}\label{es1}
\|\varphi^{\pm}_{\zeta,\tau}\|_{L^{p'}(\Omega)}\le \mathbf{c}_0\|\zeta\|_{H^2(\Omega)}.
\end{equation}
Furthermore,  as $\varphi^{\pm}_{\zeta,\tau} =e^{ i\sqrt{\lambda}\phi }\zeta$ and $\varphi^{\pm}_{\zeta,\tau} =e^{ i\overline{\sqrt{\lambda}}\phi }\zeta$ on $\Gamma$, we obtain
\[
\|\varphi^{\pm}_{\zeta,\tau}\|_{L^{q'}(\Gamma)}\le \mathbf{c}_0\|\zeta\|_{L^{q'}(\Gamma)}.
\]
Since $H^3(\Omega)$ is continuously embedded in $W^{1,q'}(\Omega)$ (\cite[(7.30)]{GT}) and the mapping $w\in W^{1,q'}(\Omega)\mapsto w_{|\Gamma}\in L^{q'}(\Gamma)$ is continuous, the mapping $w\in H^3(\Omega)\mapsto w_{|\Gamma}\in L^{q'}(\Gamma)$ is also continuous. Therefore, the previous inequality yields
\begin{equation}\label{es2}
\|\varphi^{\pm}_{\zeta,\tau}\|_{L^{q'}(\Gamma)}\le \mathbf{c}_0\|\zeta\|_{H^3(\Omega)}.
\end{equation} 
Consequently, we obtain
\begin{equation}\label{es3}
\|\varphi^{\pm}_{\zeta,\tau}\|\le \mathbf{c}_0\tau^2\|\zeta\|_{H^3(\Omega)}.
\end{equation}

\subsection{Main result}

We fix $0<\theta<1/2$. Let $r=\frac{2n}{n+2\theta}$ and $s>1$ chosen such that
\begin{equation}\label{ghi}
\frac{1}{r}=\frac{1}{s}+\frac{1}{p'}.
\end{equation}
That is we have
\[
s=\frac{n}{\theta+1}\in \left]\frac{2n}{3},n\right[.
\]

Define
\begin{align*}
&\mathscr{V}=\left\{V\in  L^{s}(\Omega,\mathbb{R});\; |V|\le V_0\right\},
\\
&\mathscr{W}:=\left\{(V_1,V_2)\in \mathscr{V}\times \mathscr{V};\; |V_1-V_1|\le W_0 \right\},
\end{align*}
where $V_0\in L^s(\Omega)$ and $W_0\in L^\mathfrak{a}(\Omega)$ are nonnegative and not identically zero ($\mathfrak{a}$ is defined in the introduction).

Taking \eqref{ghi} into account, and applying Hölder's inequality, we obtain from \eqref{es1}
\begin{equation}\label{es1.1}
\|V\varphi^{\pm}_{\zeta,\tau}\|_{L^r(\Omega)}\le \|V\|_{L^s(\Omega)}\|\varphi^{\pm}_{\zeta,\tau}\|_{L^{p'}(\Omega)}\le \mathbf{c}\|\zeta\|_{H^2(\Omega)}.
\end{equation}
Here and in the remainder of this section, $\mathbf{c}>0$ will denote a generic constant depending on $n$, $\Omega$, $\mathbf{g}$, $V_0$, $W_0$, and $\theta$.

Let $V_1,V_2\in \mathscr{V}$ and set
\[
\varphi_\tau^+:=\varphi^+_{a,\tau},\quad \varphi_\tau^-:=\varphi^-_{b,\tau}.
\]

Let $d\mu$ and $ds$ denote the Riemannian measure on $\Omega$ and $\Gamma$, respectively. We will use the following integral identity whose proof is given in Appendix \ref{app3}.
\begin{align}
&\int_\Gamma [\Lambda_{V_2}(\lambda)-\Lambda_{V_1}(\lambda)](\varphi_\tau^+)\overline{\varphi_\tau^-}ds=\int_\Omega (V_2-V_1)ab d\mu \label{ii2}
\\
&\qquad-\int_\Omega R_{V_1}(\lambda)(e^{i\sqrt{\lambda}\phi}\Delta_{V_1} a)\overline{\varphi_\tau^-}d\mu+\int_\Omega R_{V_2}(\lambda)(e^{i\sqrt{\lambda}\phi}\Delta_{V_2}a)\overline{\varphi_\tau^-}d\mu,\nonumber
\end{align}
where $\Delta_{V_j}:=\Delta -V_j$, $j=1,2$. 

As above, by interpolation, we have
\[
\|V_ja\|_{L^r(\Omega)}\le \|V_j\|_{L^s(\Omega)}\|a\|_{L^{p'}(\Omega)} \le  \mathbf{c}\|a\|_{H^2(\Omega)},\quad j=1,2.
\]
This, together with \eqref{1.21} and \eqref{es1}, imply
\begin{align}
&\left|\int_\Omega R_{V_1}(\lambda)(e^{i\sqrt{\lambda}\phi}\Delta_{V_1} a)\overline{\varphi_\tau^-}d\mu\right|+\left|\int_\Omega R_{V_2}(\lambda)(e^{i\sqrt{\lambda}\phi}\Delta_{V_2}a)\overline{\varphi_\tau^-}d\mu\right| \label{m1}
\\
&\hskip 6.5cm \le \mathbf{c}\tau^{-\vartheta}\|a\|_{H^2(\Omega)}\|b\|_{H^2(\Omega)},\nonumber
\end{align}
where $\vartheta=1-2\theta$.

From \eqref{ii2} and \eqref{m1}, we derive
\begin{align}
&\left|\int_\Omega(V_1-V_2)abds\right|\le \|(\Lambda_{V_1}(\lambda_\tau)-\Lambda_{V_2}(\lambda_\tau))(\varphi_\tau^+)\|_{L^q(\Gamma)}\|\varphi_\tau^-\|_{L^{q'}(\Gamma)}\label{m2}
\\
&\hskip 6.5cm+\mathbf{c}\tau^{-\vartheta}\|a\|_{H^2(\Omega)}\|b\|_{H^2(\Omega)}. \nonumber
\end{align}

In addition, assume that $(V_1,V_2)\in  \mathscr{W}$. Then we obtain from \eqref{2.13}
\[
\mathbf{c} \|\Lambda_{V_1}(\lambda_\tau)(\varphi_\tau^+)-\Lambda_{V_2}(\lambda_\tau)(\varphi_\tau^+)\|\le (\tau^{-3} +\tau^{\sigma(\ell +1)+n+2}\mathscr{D}(V_1,V_2))\|\varphi_\tau^+\|. 
\]
This and \eqref{es3} give
\begin{equation}\label{m4}
\mathbf{c} \|\Lambda_{V_1}(\lambda_\tau)(\varphi_\tau^+)-\Lambda_{V_1}(\lambda_\tau)(\varphi_\tau^+)\|\le (\tau^{-1} +\tau^{\sigma(\ell +1)+n+4}\mathscr{D}(V_1,V_2))\|a\|_{H^3(\Omega)}. 
\end{equation}
By combining \eqref{m2} and \eqref{m4}, we find
\begin{align}
&\mathbf{c}\left|\int_\Omega(V_1-V_2)ab ds\right| \label{m5}
\\
&\hskip 1.5cm \le (\tau^{-1}+ \tau^{\sigma(\ell +1)+n+4}\mathscr{D}(V_1,V_2))\|b\|_{H^3(\Omega)}\|a\|_{H^3(\Omega)}\nonumber
\\
&\hskip 6.5cm+\tau^{-\vartheta}\|a\|_{H^2(\Omega)}\|b\|_{H^2(\Omega)}.\nonumber
\end{align}

We now specify $a$ and $b$. Let $\eta\in H^3(\partial_+S\Omega_1)$ be arbitrarily chosen, $a=\rho^{-\frac{1}{4}}$ and $b=\rho^{-\frac{1}{4}}\eta$.

Recall that $H_0^2(\Omega)$ is the closure of $C_0^\infty (\Omega)$ in $H^2(\Omega)$, and suppose, in addition, that $V_1-V_2\in H_0^2(\Omega)$ and
\[
\|V_1-V_2\|_{H^2(\Omega)}\le \varpi,
\]
for some $\varpi >0$.  

Let $V:=\chi_\Omega(V_1-V_2)\in H^2(\Omega_1)$. Then, we have
\begin{align*}
\int_\Omega Vab ds&=\int_{S_{\tilde{x}}^+\Omega_1}\int_0^{\ell_+(\tilde{x},\theta)}V(r,\theta)\eta (\tilde{x},\theta)drd\omega_{\tilde{x}}(\theta)
\\
&=\int_{S_{\tilde{x}}^+\Omega_1}\mathbf{I}(V)(r,\theta)\eta (\tilde{x},\theta)d\omega_{\tilde{x}}(\theta),
\end{align*}
which, combined with \eqref{m5}, gives
\begin{align}
&\left|\int_{\partial_+S\Omega_1}\mathbf{I}(V)(r,\theta)\eta (\tilde{x},\theta)ds(\tilde{x})d\omega_{\tilde{x}}(\theta)\right| \label{m5.1}
\\
&\hskip 4cm \le \mathbf{c}(\tau^{-\vartheta}+ \tau^{\sigma(\ell +1)+n+4}\mathscr{D}(V_1,V_2))\|\eta\|_{H^3(\Omega)}.\nonumber
\end{align}
Given that $V\in H^2(\Omega_1)$, we have $\mathbf{N}(V)\in H^3(\Omega_1)$. Therefore, we can choose in \eqref{m5.1}
\[
\eta (\tilde{x},\theta)=\mathbf{I}\mathbf{N}(V)(\tilde{x},\theta)|\langle \theta,\nu(x)\rangle|,
\]
which belongs to $H^3(\partial_+S\Omega_1)$. With this choice, we have
\begin{equation}\label{m5.2}
\int_{\partial_+S\Omega_1}\mathbf{I}(V)(r,\theta)\eta (\tilde{x},\theta)ds(\tilde{x})d\omega_{\tilde{x}}(\theta)=\|\mathbf{N}(V)\|
_{L^2(\Omega_1)}^2.
\end{equation}
On the other hand
\[
\|\eta\|_{H^3(\partial_+S\Omega_1)}\le \tilde{c}\|\mathbf{N}(V)\|_{H^3(\Omega_1)},
\]
where $\tilde{c}=\tilde{c}(\Omega_1)>0$ is a constant. In view of \eqref{nor2}, we then obtain
\begin{equation}\label{m5.3}
\|\eta\|_{H^3(\partial_+S\Omega_1)}\le \tilde{c}c_2\|V\|_{H^2(\Omega_1)}=\tilde{c}c_2\|V_1-V_2\|_{H^2(\Omega)}\le \tilde{c}c_2\varpi.
\end{equation}
From now on, the constant $\mathbf{c}$ also depends on $\varpi$. Then, using \eqref{m5.2} and \eqref{m5.3} in \eqref{m5.1}, we get
\begin{equation}\label{m7.0}
\mathbf{c}\|\mathbf{N}(V_1-V_2)\|_{L^2(\Omega_1)}^2\le\tau^{\sigma(\ell +1)+n+4}\mathscr{D}(V_1,V_2)
+\tau^{-\vartheta}.
\end{equation}
On the other hand, we have from \eqref{nor1}
\begin{align*}
c_0\|V_1-V_2\|_{L^2(\Omega)}&\le \|\mathbf{N}(V_1-V_2)\|_{H^1(\Omega_1)}
\\
&\le \mathbf{c}\|\mathbf{N}(V_1-V_2)\|_{L^2(\Omega_1)}^{1/2}\|\mathbf{N}(V_1-V_2)\|_{H^2(\Omega_1)}^{1/2}
\\
&\le \mathbf{c}\|\mathbf{N}(V_1-V_2)\|_{L^2(\Omega_1)}^{1/2}.
\end{align*}
This and \eqref{m7.0} imply
\begin{equation}\label{m7.1}
\mathbf{c}\|V_1-V_2\|_{L^2(\Omega)}^4\le \tau^{\sigma(\ell +1)+n+6}\mathscr{D}(V_1,V_2)
+\tau^{-\vartheta}.
\end{equation}

Define
\[
\tilde{\mathscr{W}}=\{(V_1,V_2)\in \mathscr{W};\; V_1-V_2\in H_0^2(\Omega),\, \|V_1-V_2\|_{H^2(\Omega)}\le \varpi\}.
\]
By minimizing the right hand side of \eqref{m7.1} with respect to $\tau$, we obtain the following result, where
\[
\tilde{\beta}:=\frac{\vartheta}{4(\sigma(\ell +1)+n+4+\vartheta)}.
\]

\begin{theorem}\label{mainthm2}
For all $(V_1,V_2)\in \tilde{\mathscr{W}}$  we have
\[
\|V_1-V_2\|_{L^2(\Omega)}\le \mathbf{c}\mathscr{D}(V_1,V_2)^{\tilde{\beta}}.
\]
\end{theorem}

\appendix

\section{Coercivity of the sesquilinear form $\mathfrak{a}_V$}\label{app1}

Before demonstrating the coercivity of $\mathfrak{a}_V$, we establish the following preliminary lemma. We recall, for convenience, that
\begin{align*}
& \mathbf{e}=\sup\left\{\|u\|_{L^{p'}(\Omega)};\; u\in H^1(\Omega), \; \|u\|_{H^1(\Omega)}=1\right\}.
\\
&\mathbf{p}=\sup \left\{ \|u\|_{H^1(\Omega )};\; u\in H_0^1(\Omega),\; \|\nabla u\|_{L^2(\Omega )}=1\right\}.
\end{align*}

\begin{lemma}\label{lem1}
Let $V\in L^{n/2}(\Omega )$ and denote by $\mathfrak{b}_V$ the sesquilinear  form defined by
\[
\mathfrak{b}_V(u,v)=\int_\Omega Vu\overline{v}dx,\quad u,v\in H^1(\Omega ).
\]
Then $\mathfrak{b}$ is bounded with
\begin{equation}\label{1.1}
|\mathfrak{b}_V(u,v)|\le\mathbf{e}^2 \|V\|_{L^{n/2}(\Omega )}\|u\|_{H^1(\Omega )}\|v\|_{H^1(\Omega )},\quad u,v\in H^1(\Omega ).
\end{equation}
Moreover, for all $u\in H^1(\Omega)$, $\ell_V(u):v\in H_0^1(\Omega )\mapsto \mathfrak{b}_V(u,v)$ belongs to $H^{-1}(\Omega )$ with
\begin{equation}\label{1.2}
\|\ell_V(u)\|_{H^{-1}(\Omega )}\le \mathbf{e}^2 \|V\|_{L^{n/2}(\Omega )}\|u\|_{H^1(\Omega )}.
\end{equation}
\end{lemma}

\begin{proof}
Let $u,v\in H^1(\Omega )$. By applying Cauchy-Schwarz's inequality, we get 
\begin{equation}\label{1.3}
|\mathfrak{b}_V(u,v)|\le \|Vu^2\|_{L^1(\Omega )}^{1/2}\|Vv^2\|_{L^1(\Omega )}^{1/2}.
\end{equation}
It follows from H\"older's inequality 
\begin{equation}\label{1.4}
\|Vw^2\|_{L^1(\Omega )}^{1/2}\le \|V\|_{L^{n/2}(\Omega)}^{1/2}\|w\|_{L^{p'}(\Omega )}\le \mathbf{e} \|V\|_{L^{n/2}(\Omega)}^{1/2}\|w\|_{H^1(\Omega )},\quad w\in\{u,v\}.
\end{equation}
Then \eqref{1.4}  in \eqref{1.3} yields \eqref{1.1}.

Let $u\in H^1(\Omega )$. From \eqref{1.1}, we obtain
\[
|\ell_V(u)(v)|\le \mathbf{e}^2 \|V\|_{L^{n/2}(\Omega )}\|u\|_{H^1(\Omega )}\|v\|_{H_0^1(\Omega )},\quad v\in H_0^1(\Omega ).
\]
Therefore,  $\ell_V(u)\in H^{-1}(\Omega )$ and \eqref{1.2} holds. 
\end{proof}

Let $V\in L^{n/2}(\Omega,\mathbb{R} )$ and recall that $\mathfrak{a}_V$ is given by
\begin{align*}
\mathfrak{a}_V(u,v)&=\int_\Omega \left(\nabla u\cdot \nabla\overline{ v}+Vu\overline{v}\right)dx
\\
&=\int_\Omega \nabla u\cdot \nabla\overline{v}dx+\mathfrak{b}_V(u,v), \quad u,v\in H_0^1(\Omega).
\end{align*}

In light of Lemma \ref{lem1}, $\mathfrak{a}_V$ is bounded with
\[
|\mathfrak{a}_V(u,v)|\le \left(1+\|V\|_{L^{n/2}(\Omega )}\right)\|u\|_{H^1(\Omega )}\|v\|_{H^1(\Omega )}.
\]
As $C_0^\infty (\Omega )$ is dense in $L^{n/2}(\Omega )$, we find $\tilde{V}\in C_0^\infty(\Omega )$ such that 
\[
\|V-\tilde{V}\|_{L^{n/2}(\Omega )}\le \frac{1}{2\mathbf{p}^2 \mathbf{e}^2}.
\]
Therefore, we have for $u\in H_0^1(\Omega)$
\begin{align*}
\mathfrak{a}_V(u,u)&\ge \mathbf{p}^{-2} \|u\|_{H^1(\Omega)}^2-\|\tilde{V}\|_{L^\infty (\Omega )}\|u\|_{L^2(\Omega )}^2
\\
&\hskip 3cm -\|V-\tilde{V}\|_{L^{n/2}(\Omega )}\|u\|_{L^{p'}(\Omega)}^2
\\
&\ge \frac{\mathbf{p}^{-2}}{2}\|u\|_{H^1(\Omega)}^2-\|\tilde{V}\|_{L^\infty (\Omega )}\|u\|_{L^2(\Omega )}^2.
\end{align*}
This shows that $\mathfrak{a}_V$ is coercive.

\section{BVPs with unbounded lower order term}\label{app2}

We will use the following lemma, the proof of which is quite similar to that of Lemma \ref{lem4}.

\begin{lemma}\label{lemB4}
Let $V\in  \mathscr{V}$, $\lambda\in \rho(A_V)$. Then
\begin{align}
&\|R_V(\lambda)f\|_{L^2(\Omega)}\le \frac{1}{\mathrm{dist}(\lambda,\sigma(A_V))} \|f\|_{L^2(\Omega)}, \quad f\in L^2(\Omega), \label{B7}
\\
&\|R_V(\lambda)f\|_{H_0^1(\Omega)}\le \mathbf{c}_\lambda \|f\|_{L^2(\Omega)},\quad f\in L^2(\Omega),\label{B8}
\\
&\|R_V(\lambda)f\|_{H_0^1(\Omega)}\le \mathbf{c}_\lambda \|f\|_{L^p(\Omega)},\quad \quad f\in L^p(\Omega).\label{B9}
\end{align}
Here, $\mathbf{c}_\lambda=\mathbf{c}_\lambda (n,\Omega,V,\lambda)>0$ is a constant.
\end{lemma}

We consider the  BVP
\begin{equation}\label{1.7}
\left\{
\begin{array}{ll}
(-\Delta +V- \lambda)u=f\quad &\mbox{in}\; \Omega,
\\
u=\varphi &\mbox{on}\; \Gamma .
\end{array}
\right.
\end{equation}

\begin{theorem}\label{thm2}
Let $V\in L^{n/2}(\Omega,\mathbb{R} )$ and $\lambda \in \rho (A_V)$. For all $f\in L^p(\Omega)$ and $\varphi\in W^{2-1/p,p}(\Gamma )$ the BVP \eqref{1.7} has a unique solution $u=u_V(\lambda)(\varphi ,f)\in W^{2,p}(\Omega )$. In addition, there exists a constant $\mathbf{c}_\lambda=\mathbf{c}_\lambda (n,\Omega,V,\lambda)>0$ such that
\begin{equation}\label{1.8}
\|u_V(\lambda)(\varphi,f)\|_{W^{2,p}(\Omega )}\le \mathbf{c}_\lambda\left(\|f\|_{L^p(\Omega)}+\|\varphi\|_{W^{2-1/p,p}(\Gamma )}\right).
\end{equation}
\end{theorem}

\begin{proof}
In this proof, $\mathbf{c}_0=\mathbf{c}_0(n,\Omega)>0$, $\mathbf{c}=\mathbf{c}(n,\Omega,V)>0$ and $\mathbf{c}_\lambda=\mathbf{c}_\lambda (n,\Omega,V,\lambda)>0$ are generic constants.

Let $F\in W^{2,p}(\Omega )$ be such that $\gamma_0F=\varphi$ and
\[
\|F\|_{W^{2,p}(\Omega )}\le \mathbf{c}_0 \|\varphi\|_{W^{2-1/p,p}(\Gamma )}.
\]
If  $G=-(-\Delta +V- \lambda)F+f$ then $G\in L^p(\Omega )$ and
\begin{equation}\label{1.9}
\|G\|_{L^p(\Omega )}\le \mathbf{c}_\lambda\left(\|f\|_{L^p(\Omega)}+\|\varphi\|_{W^{2-1/p,p}(\Gamma )}\right).
\end{equation}
It follows from \eqref{B9} that $v=R_V(\lambda )G\in H_0^1(\Omega)$ and 
\begin{equation}\label{1.10}
\|v\|_{H_0^1(\Omega)}\le \mathbf{c}_\lambda \|G\|_{L^p(\Omega)}.
\end{equation}
Furthermore, we have
\[
-\Delta v= -Vv+\lambda v+G \in L^p(\Omega ).
\]

By \cite[Theorem 2.4.2.6, page 241]{Gr}, we obtain $v\in W^{2,p}(\Omega )$. Then we proceed as in the proof of Lemma \ref{lemres} to obtain
\[
\|v\|_{W^{2,p}(\Omega )}\le \mathbf{c}\|(-\Delta +\lambda_0)v\|_{L^p(\Omega )},
\]
where $\lambda_0=\lambda_0(n,\Omega)>0$ is a constant, and consequently
\begin{align*}
\|v\|_{W^{2,q}(\Omega )}&\le \mathbf{c}_0\|-Vv+(\lambda+\lambda_0) v+G\|_{L^p(\Omega )}
\\
&\le \mathbf{c}_0\left(  \|V_0\|_{L^{n/2}(\Omega )}\|v\|_{L^{p'}(\Omega )}+|\lambda +\lambda_0|\|v\|_{L^p(\Omega )}+\|G\|_{L^p(\Omega )}\right).
\end{align*}
The last inequality, combined with \eqref{1.10}, gives
\[
\|v\|_{W^{2,p}(\Omega )}\le \mathbf{c}_\lambda\|G\|_{L^p(\Omega )}.
\]
That is, we have
\begin{equation}\label{1.11}
\|v\|_{W^{2,p}(\Omega )}\le \mathbf{c}_\lambda(\|f\|_{L^p(\Omega)}+\|\varphi\|_{W^{2-1/p,p}(\Gamma )}).
\end{equation}

We verify that $u=F+v\in W^{2,p}(\Omega )$ is a solution of the BVP \eqref{1.7} and inequality \eqref{1.8}  follows from \eqref{1.11}. This solution is unique because $\lambda$ is not an eigenvalue of $ A_V$. 
\end{proof}

\section{Integral identity}\label{app3}

For simplicity, we use the notations $\varphi_1:= \varphi_{a,\tau}^+$, $\varphi_2=\varphi_{b,\tau}^-$ and recall that $\Delta_{V_j}=\Delta -V_j$, $j=1,2$, and $\lambda=(\tau+i)^2$ with $\tau \ge 1$.

The following Green's formula will be used in the sequel, where $V=V_1$ or $V=V_2$,
\begin{equation}\label{mgf}
\int_\Omega(\Delta_V+\lambda)u\overline{v}d\mu=\int_\Omega u\overline{(\Delta_V+\overline{\lambda})v}d\mu+\int_\Gamma \partial_{\nu}u\overline{v}ds-\int_\Gamma u\overline{\partial_{\nu}v}ds.
\end{equation}

Let 
\[
u_1:=u_{V_1}(\lambda )(\varphi_1)=\varphi_1+R_{V_1}(\lambda)((\Delta_{V_1}+\lambda)\varphi_1).
\]
In light of \eqref{u5}, we have
\begin{equation}\label{u7}
u_1=\varphi_1+R_{V_1}(\lambda)(e^{i\sqrt{\lambda}\phi}\Delta_{V_1} a).
\end{equation}

By applying Green's formula \eqref{mgf} with $u=u_1$ and $v=\varphi_2$, we obtain
\[
-\int_\Gamma \Lambda_{V_1}(\lambda)(\varphi_1)\overline{\varphi_2}ds+\int_\Gamma \varphi_1\overline{\partial_{\nu}\varphi_2}ds =\int_\Omega u_1\overline{(\Delta_{V_1}+\overline{\lambda})\varphi_2}d\mu.
\]
As $(\Delta_{V_1}+\overline{\lambda})\varphi_2=-V_1\varphi_2$, we get
\[
-\int_\Gamma \Lambda_{V_1}(\lambda)(\varphi_1)\overline{\varphi_2} ds+\int_\Gamma \varphi_1\overline{\partial_{\nu}\varphi_2}ds =-\int_\Omega u_1V_1\overline{\varphi_2}d\mu.
\]
This and \eqref{u7} imply
\begin{align}
&-\int_\Gamma \Lambda_{V_1}(\lambda)(\varphi_1)\overline{\varphi_2}ds+\int_\Gamma \varphi_1\overline{\partial_{\nu}\varphi_2}ds=-\int_\Omega \varphi_1V_1\overline{\varphi_2}d\mu \label{u8}
\\
&\hskip 4cm -\int_\Omega R_{V_1}(\lambda)(e^{i\sqrt{\lambda}\phi}\Delta_{V_1} a)\overline{\varphi_2}d\mu.\nonumber
\end{align}

Now, let 
\begin{align*}
v_1:=u_{V_2}(\lambda)(\varphi_1)&=\varphi_1+R_{V_2}(\lambda)((\Delta_{V_2}+\lambda)\varphi_1)
\\
&= \varphi_1+R_{V_2}(\lambda)(e^{i\sqrt{\lambda}\phi}\Delta_{V_2}a).
\end{align*}
Once again, Green's formula \eqref{mgf} with $u=v_1$ and $v=\varphi_2$ yields
\[
-\int_\Gamma \Lambda_{V_2}(\lambda)(\varphi_1)\overline{\varphi_2}ds+\int_\Gamma \varphi_1\overline{\partial_{\nu}\varphi_2}ds =\int_\Omega v_1\overline{(\Delta_{V_2}+\overline{\lambda})\varphi_2}d\mu.
\]
In view of \eqref{u5}, this inequality gives
\[
-\int_\Gamma \Lambda_{V_2}(\lambda)(\varphi_1)\overline{\varphi_2}ds+\int_\Gamma \varphi_1\overline{\partial_{\nu}\varphi_2}ds =-\int_\Omega v_1V_2\overline{\varphi_2}d\mu.
\]
Hence,
\begin{align}
&-\int_\Gamma \Lambda_{V_2}(\lambda)(\varphi_1)\overline{\varphi_2}ds+\int_\Gamma \varphi_1\overline{\partial_{\nu}\varphi_2}ds=-\int_\Omega \varphi_1V_2\overline{\varphi_2}d\mu \label{u9}
\\
&\hskip 4cm -\int_\Omega R_{V_2}(\lambda)(e^{i\sqrt{\lambda}\phi}\Delta_{V_2}a)\overline{\varphi_2}d\mu.\nonumber
\end{align}
By taking the difference, side by side, of \eqref{u8} and \eqref{u9}, we find
\begin{align*}
&\int_\Gamma [\Lambda_{V_2}(\lambda)-\Lambda_{V_1}(\lambda)](\varphi_1)\overline{\varphi_2}ds=\int_\Omega (V_2-V_1)\varphi_1\overline{\varphi_2}d\mu 
\\
&\qquad-\int_\Omega R_{V_1}(\lambda)(e^{i\sqrt{\lambda}\phi}\Delta_{V_1} a)\overline{\varphi_2}d\mu+\int_\Omega R_{V_2}(\lambda)(e^{i\sqrt{\lambda}\phi}\Delta_{V_2}a)\overline{\varphi_2}d\mu,
\end{align*}
which, combined with the fact that $\varphi_1\overline{\varphi_2}=ab$, yields
\begin{align}
&\int_\Gamma [\Lambda_{V_2}(\lambda)-\Lambda_{V_1}(\lambda)](\varphi_1)\overline{\varphi_2}ds=\int_\Omega (V_2-V_1)ab d\mu \label{u10}
\\
&\qquad-\int_\Omega R_{V_1}(\lambda)(e^{i\sqrt{\lambda}\phi}\Delta_{V_1} a)\overline{\varphi_2}d\mu+\int_\Omega R_{V_2}(\lambda)(e^{-i\sqrt{\lambda}\phi}\Delta_{V_2}a)\overline{\varphi_2}d\mu.\nonumber
\end{align}
The proof is then complete.

\medskip
\noindent
\textbf{Declarations:}

\smallskip
\noindent
\textbf{Conflict of interest.} The author declares that he has no conflict of interest.

\smallskip
\noindent
\textbf{Data availability.} Data sharing not applicable to this article as no datasets were
generated or analyzed during the current study.

\end{document}